\documentclass[11pt]{article}
\usepackage[T1]{fontenc}
\usepackage[utf8]{inputenc}
\usepackage{lmodern,amsmath,amssymb,amsthm,mathtools}
\usepackage[margin=1in]{geometry}
\usepackage{microtype,booktabs,enumitem}
\usepackage[colorlinks=true,linkcolor=blue,citecolor=blue,urlcolor=blue]{hyperref}
\hypersetup{pdftitle={Perfect state transfer on mixed graphs: complete classes and transfer times},pdfauthor={Xingkun Song}}
\numberwithin{equation}{section}
\newtheorem{theorem}{Theorem}[section]
\newtheorem{lemma}[theorem]{Lemma}
\newtheorem{proposition}[theorem]{Proposition}
\newtheorem{corollary}[theorem]{Corollary}
\theoremstyle{definition}

\newtheorem{question}{Question}
\theoremstyle{remark}

\newcommand{\ii}{\mathrm{i}}
\newcommand{\R}{\mathbb R}
\newcommand{\T}{\mathcal T}
\newcommand{\e}{\mathbf e}
\newcommand{\one}{\mathbf 1}
\DeclareMathOperator{\diag}{diag}
\title{Perfect state transfer on mixed graphs:\\complete classes and transfer times}
\author{Xingkun Song$^{1,2}$\\
{\small $^{1}$ School of Mathematics and Statistics, Qinghai Minzu University,}\\[-2pt]
{\small Xining, Qinghai 810007, P.R. China}\\
{\small $^{2}$ Qinghai Institute of Applied Mathematics,}\\[-2pt]
{\small Xining, Qinghai 810007, P.R. China}\\
{\small \href{mailto:xksong@126.com}{xksong@126.com}}}
\date{}
\begin{document}
\maketitle
\begin{abstract}
For perfect state transfer (PST) on unweighted mixed graphs, we classify the normalized transfer times of complete PST classes. A finite set $\Lambda\subset\mathbb R/\mathbb Z$ containing zero occurs at a nonstationary periodic vertex if and only if $\cos(2\pi(x-y))\in\mathbb Q$ for all $x,y\in\Lambda$. Every admissible set has a connected oriented realization. We also classify the possible return phases of oriented realizations at the minimum vertex period. Transfers at rational multiples of the common minimum vertex period partition a complete class into sets of size at most six, or at most three in an oriented graph with return phase $-1$; both bounds are sharp. We construct complete classes of every finite size, including classes in which all transfers between distinct vertices occur at irrational multiples of the period and no switching automorphism maps a class vertex to a distinct class vertex. We also characterize simultaneous realization in connected oriented graphs with prescribed relative minimum vertex periods and return phases. The proof combines an imaginary quadratic field restriction with an unweighted construction that selects the complete target set.
\end{abstract}

\noindent\textbf{Keywords.} Perfect state transfer; mixed graph; oriented graph; equitable partition; vertex periodicity.
\par\smallskip
\noindent\textbf{Mathematics Subject Classification (2020).} 05C50, 81P45, 15A18.

\section{Introduction}
Unless stated otherwise, all graphs are finite, simple, unweighted mixed graphs. An oriented graph is a mixed graph with no undirected edges. Connectivity and vertex degrees refer to the underlying undirected graph.

Let $H$ be the Hermitian adjacency matrix of a mixed graph and put $U_H(t)=e^{-\ii tH}$. Perfect state transfer (PST) occurs from $a$ to $b$ at time $t$ if $U_H(t)\e_a=\gamma\e_b$, where $\e_a$ is the standard basis vector at $a$ and $|\gamma|=1$. In an oriented graph, $H=\ii A$ for a real skew-symmetric matrix $A$, so $U_H(t)=e^{tA}$. The first transfer times in the two directions may differ, and a vertex may have PST to several other vertices at different times.

The \emph{complete PST class} of $a$ consists of all vertices reached from $a$ by PST at real times, including $a$ itself. For Hermitian matrices with algebraic entries, PST between distinct vertices implies periodicity by \cite[Theorem~6.1]{GReal}. Every nontrivial complete class therefore has a common minimum vertex period $T$: the first positive time when a vertex state returns to its initial coordinate line, allowing a phase factor. For a fixed source $a$, a transfer to $b$ at time $t$ defines its \emph{normalized transfer time} $x_b=t/T\pmod1$. Distinct targets have distinct normalized transfer times. We write $\Lambda$ for their set, so $0\in\Lambda$. Singleton classes are included when their source is nonstationary and periodic; nonstationary means that its evolution does not remain on its initial coordinate line.

Godsil and Lato~\cite{GL} developed the spectral theory of PST on oriented graphs. Acuaviva, Chan, Eldridge, Godsil, How-Chun-Lun, Tamon, Wright and Zhang~\cite{AC} constructed graphs with several mutually transferring vertices, and Miki, Tsujimoto and Zhao~\cite{MTZ} studied multiple state transfer for weighted Hermitian Hamiltonians. Under the additional assumption that the skew adjacency matrix belongs to a Bose--Mesner algebra, Chan, Pantangi, Razafimahatratra and Sin~\cite[Corollary~2.3]{CPRS} proved that a nontrivial complete class has size $2,3,4$, or $6$. This includes normal oriented Cayley graphs. We impose no such symmetry assumption and ask which complete classes and transfer times are possible in general.

\begin{question}\label{prob:classification}
Which finite sets $\Lambda\subset\mathbb R/\mathbb Z$ containing zero occur as the sets of normalized transfer times of complete PST classes at nonstationary periodic vertices in connected mixed graphs?
\end{question}

Realizing the prescribed transfers alone does not answer this question: every other possible target must also be excluded. The following theorem gives the classification. In an oriented graph every transfer and return phase is $+1$ or $-1$.

\begin{theorem}\label{thm:schedules}
Let $\Lambda\subset\mathbb R/\mathbb Z$ be a finite set containing zero. The following statements are equivalent.
\begin{enumerate}[label=\textup{(\roman*)},leftmargin=*]
\item $\Lambda$ is the set of normalized transfer times of a complete PST class at a nonstationary periodic vertex in a connected graph.
\item The same holds in a connected oriented graph with return phase $+1$ at the minimum vertex period.
\item There is an imaginary quadratic field $J$ such that $e^{2\pi\ii x}\in J$ for every $x\in\Lambda$.
\item $\cos(2\pi(x-y))\in\mathbb Q$ for all $x,y\in\Lambda$.
\end{enumerate}
An oriented realization with return phase $-1$ at the minimum vertex period exists if and only if
\begin{equation}\label{eq:negativephasecondition}
 \cos(\pi(x-y))\in\mathbb Q\qquad(x,y\in\Lambda),
\end{equation}
using representatives in $[0,1)$. Rationality in this condition is independent of the representatives.
\end{theorem}

The theorem prescribes transfer times relative to the minimum vertex period, rather than its absolute value. Differences in $\Lambda$ determine all transfers within the class. For example, the set $\{0,1/3\}$ gives first transfer times $T/3$ and $2T/3$ in the two directions. Allowing undirected edges does not enlarge the family of admissible sets. For oriented realizations, either admissible return phase can be combined with prescribed source-to-target signs; the formula for all times and phases is~\eqref{eq:allphases}, at the end of the proof in Section~\ref{sec:unbounded}.

The classification gives complete classes of every finite size. It answers the existence question for a six-vertex class in an oriented graph posed by Chan, Pantangi, Razafimahatratra and Sin~\cite[Section~1]{CPRS}, and the question of transfer among more than four vertices posed by Song and Lin~\cite[Problem~36]{SL}. Theorem~\ref{thm:rationalblocks} describes all partitions defined by transfer at rational multiples of the minimum vertex period: the sharp bound on each part is six, or three for oriented graphs with return phase $-1$. Theorem~\ref{thm:simultaneous} characterizes simultaneous realization of several classes with prescribed relative minimum vertex periods and return phases.

Godsil and Lato~\cite[Section~8]{GL} also asked whether multiple state transfer can occur without switching automorphisms relating the transferring vertices, and how the first transfer time compares with the period in that case. Their Figure~3 gives a transferring pair with an irrational time ratio. Corollary~\ref{thm:main} gives complete classes of every size $k\geq2$ in which all transfers between distinct vertices occur at irrational multiples of the minimum vertex period and no switching automorphism maps a class vertex to a distinct class vertex. For $k\geq3$, this answers their multiple-state-transfer question.

The proof has three steps. First, an arithmetic argument places the unit-circle points corresponding to the normalized transfer times in one imaginary quadratic field. Second, a rational cosine polynomial supplies spectral weights for a rational skew-symmetric matrix with the prescribed transfers. Third, Theorem~\ref{thm:selection} removes the integer weights and selects exactly the required targets. It applies to any selected subset of a connected integer-weighted mixed graph, retains precisely its original PST relations among the selected vertices, and divides their transfer and return times by one common factor. It preserves orientation; odd Cartesian powers also preserve transfer and return phases in the oriented case.

Orthogonal transformations and equitable partitions are established tools: Bachman, Fredette, Fuller, Landry, Opperman, Tamon and Tollefson~\cite{BFF} used quotient graphs and equitable partitions, and Serra, Ferr\'on and Osenda~\cite{SFO} used Givens transformations for weighted spin networks. Here the cell sizes remove the weights and determine the complete target set. The realizations of admissible sets guarantee periodicity at the prescribed vertices. Global periodicity is a separate requirement.

Sections~\ref{sec:prelim}--\ref{sec:unbounded} prove the classification, and Section~\ref{sec:consequences} develops its consequences. Section~\ref{sec:examples} illustrates the results with classes of sizes five and six, and with a nine-vertex globally periodic graph whose first transfer times are $T/3$ and $2T/3$. The latter answers~\cite[Problem~35]{SL} negatively even under global periodicity. Section~\ref{sec:discussion} states the remaining questions; Appendix~\ref{app:data} gives the explicit cell data.

\section{Periods and transfer classes}\label{sec:prelim}
The Hermitian adjacency matrix of a mixed graph~\cite{GM} has zero diagonal and entries in $\{0,1,\ii,-\ii\}$. An undirected edge $ab$ gives $H_{a,b}=H_{b,a}=1$. An arc from $a$ to $b$ gives $H_{a,b}=\ii$ and $H_{b,a}=-\ii$. For an oriented graph, we write $H=\ii A$, where $A$ is the real skew adjacency matrix.

We will also use explicitly weighted matrices. An \emph{integer-weighted mixed graph} has a simple underlying graph without loops, but a nonzero entry may be $w\eta$, where $w\in\mathbb Z_{>0}$ and $\eta\in\{1,\ii,-\ii\}$. The opposite entry is its complex conjugate. The general Hermitian matrices in Section~\ref{sec:arithmetic} and the rational skew-symmetric matrices in Section~\ref{sec:unbounded} are also intermediate models, not necessarily adjacency matrices of unweighted graphs.

We write $I_n$ and $O_n$ for the identity and zero matrices of order $n$, and omit the subscript when the order is clear. The vectors $\e_a$ and $\one_n$ denote a standard basis vector and the all-ones vector, respectively. We use $M^{\mathsf T}$ and $M^*$ for the transpose and conjugate transpose of a matrix $M$, respectively. The symbols $\otimes$ and $\oplus$ denote the Kronecker product and direct sum, respectively; we use the convention $I_n^{\otimes0}=1$. For vectors, $\langle\mathbf u,\mathbf v\rangle=\mathbf u^*\mathbf v$, and $\|\mathbf u\|$ is the associated norm. For a matrix, its underlying graph has an edge between distinct indices whenever the corresponding entry is nonzero.

We use the spectral notation for state transfer in~\cite{GPeriodic,GL}. For any Hermitian matrix $H$, let $E_\theta$ be the orthogonal projection onto its $\theta$-eigenspace, and write
\[
 H=\sum_\theta\theta E_\theta,\qquad
 U_H(t)=e^{-\ii tH},\qquad
 \Phi_a(H)=\{\theta:E_\theta\e_a\ne0\}.
\]
The set $\Phi_a(H)$ is the \emph{eigenvalue support} of $a$; we write $\Phi_a$ when $H$ is clear. When $H$ is the Hermitian adjacency matrix of a graph $X$, we also write $U_X$ and $\Phi_a(X)$ for $U_H$ and $\Phi_a(H)$, respectively. PST occurs from $a$ to $b$ at time $t$ if
\[
 U_H(t)\e_a=\gamma\e_b,\qquad |\gamma|=1,
\]
or equivalently $|U_H(t)_{b,a}|=1$. The scalar $\gamma$ is the phase factor. In the oriented case, $U_H(t)=e^{tA}$ is real, so $\gamma\in\{1,-1\}$.

We allow all real times when defining the \emph{complete PST class} of $a$:
\[
 \T_H(a)=\{b:|U_H(t)_{b,a}|=1\text{ for some }t\in\R\}.
\]
For a graph $X$ with Hermitian adjacency matrix $H$, we also write $\T_X(a)=\T_H(a)$. Transfer at time zero gives reflexivity, reversing time gives symmetry, and composing transfers gives transitivity. Thus the sets $\T_H(a)$ are equivalence classes.

A vertex $a$ is \emph{nonstationary} if $|\Phi_a|\geq2$. A real number $t$ is a return time at $a$ if $U_H(t)\e_a=\gamma\e_a$ for some $|\gamma|=1$. The vertex is \emph{periodic} if it has a positive return time. By the spectral decomposition, $t$ is a return time exactly when
\[
 t(\theta-\eta)\in2\pi\mathbb Z\qquad(\theta,\eta\in\Phi_a).
\]
If a nonstationary vertex is periodic, these return times form $T\mathbb Z$ for a unique $T>0$. We call $T$ the \emph{minimum vertex period}. For a periodic source $a$, adding sufficiently many periods makes any real transfer time positive. Thus $\T_H(a)$ also consists exactly of the vertices reached at positive times in this case. The scalar $\omega$ in $U_H(T)\e_a=\omega\e_a$ is the \emph{return phase at the minimum vertex period}. Unless stated otherwise, every periodic source below is nonstationary. For a general unit vector, we use the same definitions of eigenvalue support and periodicity, replacing $\e_a$ by that vector.

The following statement for vertex states is a consequence of Godsil's theorem on algebraic state transfer.

\begin{lemma}[{\cite[Theorem~6.1]{GReal}}]\label{lem:algebraicpst}
Let $H$ be a Hermitian matrix with algebraic entries. If PST occurs from $a$ to a distinct vertex $b$, then
\[
 \frac{\theta-\eta}{\theta'-\eta'}\in\mathbb Q
 \qquad(\theta,\eta,\theta',\eta'\in\Phi_a(H),\ \theta'\ne\eta').
\]
In particular, both vertices are nonstationary and periodic.
\end{lemma}

Lemma~\ref{lem:algebraicpst} ensures that the classification below applies to every nontrivial complete PST class of a mixed graph. Singleton classes are included when their source is nonstationary and periodic.

A \emph{global period} is a number $T>0$ such that $U_H(T)=\gamma I$ for some $|\gamma|=1$. The least such positive number, when it exists, is the \emph{minimum global period}. All periods in this paper allow a phase factor. Periodicity at one vertex does not imply global periodicity. For a connected graph, requiring $U_H(T)$ to be diagonal is equivalent to this definition: commutation with $H$ forces its diagonal entries to agree along every edge.

We refer to $\operatorname{span}\{\e_b\}$ as the coordinate line at $b$. Mapping one such line to another expresses PST without specifying a phase factor. The next lemma shows that differences of normalized transfer times determine all transfers within a class. In expressions involving real times, each normalized transfer time is first replaced by its representative in $[0,1)$; subsequent arithmetic in that expression is performed in $\mathbb R$.

\begin{lemma}\label{lem:labels}
Let $C=\T_H(a)$, where $a$ has minimum vertex period $T$. There are distinct normalized transfer times $x_b\in\R/\mathbb Z$, $b\in C$, with $x_a=0$, such that PST occurs from $b$ to $c$ at time $t$ if and only if
\begin{equation}\label{eq:labeldifference}
 t/T\equiv x_c-x_b\pmod1.
\end{equation}
All vertices of $C$ have the same return times, return phases, and eigenvalue support. More precisely, write $U_H(T)\e_a=\omega\e_a$ and $U_H(Tx_b)\e_a=\gamma_b\e_b$, using representatives $x_b\in[0,1)$. Then
\begin{equation}\label{eq:classphases}
 U_H\bigl(T(x_c-x_b+n)\bigr)\e_b
   =\omega^n\gamma_c\gamma_b^{-1}\e_c
 \qquad(b,c\in C,\ n\in\mathbb Z).
\end{equation}
\end{lemma}
\begin{proof}
If two times reach the line spanned by $\e_b$, their difference is a return time at $a$, hence an integer multiple of $T$. A single time cannot reach two distinct coordinate lines. These observations define unique, distinct normalized transfer times. For their representatives in $[0,1)$, write $U_H(Tx_b)\e_a=\gamma_b\e_b$. The group law gives
\[
 U_H(T(x_c-x_b))\e_b=\gamma_c\gamma_b^{-1}\e_c.
\]
Commuting a return with the transfer between $a$ and $b$ proves equality of their return times and phases, and gives~\eqref{eq:classphases}. Adding any integer multiple of $T$ to the displayed transfer therefore proves the sufficiency of~\eqref{eq:labeldifference}. Conversely, composing any transfer from $b$ to $c$ with the transfer from $a$ to $b$ gives a transfer from $a$ to $c$, proving necessity. Applying $E_\theta$ to the transfer equation proves equality of the eigenvalue supports.
\end{proof}

\section{Arithmetic restrictions}\label{sec:arithmetic}
Periodicity restricts the eigenvalue gaps: all phases must agree after one period. We first determine the possible gaps, then use them to restrict the normalized transfer times. The ratio condition and quadratic form of the eigenvalue support at a periodic vertex are classical in real and oriented quantum walks; see~\cite{GPeriodic,GL}. The version below works over any imaginary quadratic coefficient field. It requires neither an eigenvalue support symmetric about zero nor a prescribed phase factor.

Fix a squarefree positive integer $d$ and put $K=\mathbb Q(\sqrt{-d})\subset\mathbb C$. The matrices in this section are Hermitian with entries in $K$, with no integrality or zero-diagonal assumption. Mixed graphs correspond to $d=1$. We use the standard facts about splitting fields and fixed fields from~\cite[Chapters~2--3]{Milne}.

\begin{lemma}\label{lem:quadratic}
Let $H$ be a Hermitian matrix over $K$, and let $u$ be a nonstationary periodic vertex with eigenvalue support $\Phi_u$. There is a unique positive number $\delta$ such that
\begin{equation}\label{eq:gaplattice}
 \operatorname{span}_{\mathbb Z}\{\theta-\eta:\theta,\eta\in\Phi_u\}
 =\delta\mathbb Z.
\end{equation}
The minimum vertex period is $T=2\pi/\delta$. There are $q\in\mathbb Q_{>0}$ and a positive squarefree integer $\Delta$ such that
\[
 \delta=q\sqrt\Delta,\qquad
 \Phi_u\subset c+\mathbb Q\sqrt\Delta,\qquad
 c=\frac1{|\Phi_u|}\sum_{\theta\in\Phi_u}\theta\in\mathbb Q.
\]
The case $\Delta=1$ gives $\Phi_u\subset\mathbb Q$. If $\Delta>1$, the nontrivial conjugation of $\mathbb Q(\sqrt\Delta)$ sends each $\theta\in\Phi_u$ to $2c-\theta$.
\end{lemma}
\begin{proof}
Let $m_u$ be the monic polynomial of least degree satisfying $m_u(H)\e_u=0$. Linear algebra over $K$ shows that $m_u\in K[x]$. The spectral theorem gives
\[
 m_u(x)=\prod_{\theta\in\Phi_u}(x-\theta).
\]
All roots are real, so all coefficients are real. They therefore lie in $K\cap\mathbb R=\mathbb Q$. It follows that Galois automorphisms permute $\Phi_u$ and that its mean $c$ is rational.

A period $t_0$ satisfies $t_0(\theta-\eta)\in2\pi\mathbb Z$ for every pair in the support. Thus all nonzero differences have rational ratios. Their finitely generated additive group is $\delta\mathbb Z$ for a unique $\delta>0$, and the group of return times is $(2\pi/\delta)\mathbb Z$.

Since $c$ is the mean of the support, $\theta-c$ is the average of the gaps $\theta-\eta$, with $\eta\in\Phi_u$. It is therefore a rational multiple of $\delta$. Let $L$ be the splitting field of $m_u$ over $\mathbb Q$. The generator $\delta$ lies in $L$, since it is an integer combination of the eigenvalue gaps. Since $\Phi_u$ is Galois-stable, every automorphism of $L$ preserves the group in~\eqref{eq:gaplattice}. It therefore sends its generator $\delta$ to $\delta$ or $-\delta$. Thus $\delta^2\in\mathbb Q$. If $\delta$ is rational, then $\Phi_u\subset\mathbb Q$. Otherwise $\delta=q\sqrt\Delta$ with $q\in\mathbb Q_{>0}$ and squarefree $\Delta>1$. Conjugation in $\mathbb Q(\sqrt\Delta)$ fixes $c$ and negates each centered eigenvalue, as claimed.
\end{proof}

We now connect the gap $\delta$ to a transfer at time $\tau$. The unit-circle point $e^{-\ii\delta\tau}$ records the normalized transfer time. It is distinct from the phase factor $\gamma$, which will cancel in the proof. The ratio $\tau/T$ need not be rational.

\begin{lemma}\label{lem:field}
With the notation of Lemma~\ref{lem:quadratic}, suppose $U_H(\tau)\e_u=\gamma\e_v$, where $|\gamma|=1$. Define
\[
 z=e^{-\ii\delta\tau}=e^{-2\pi\ii\tau/T}.
\]
Then
\begin{equation}\label{eq:targetfield}
 z\in J:=\mathbb Q(\sqrt{-d\Delta}),
\end{equation}
with $\Delta=1$ when $\delta$ is rational.
If $H=\ii A$ with $A$ rational and real skew-symmetric, and the return phase at $T$ is $-1$, then also
\begin{equation}\label{eq:halftimefield}
 e^{-\pi\ii\tau/T}\in J.
\end{equation}
\end{lemma}
\begin{proof}
PST gives equal eigenvalue supports at $u$ and $v$. Set $L=K(\sqrt\Delta)$, so $L=K$ when $\Delta=1$. For each $\theta\in\Phi_u$, Lagrange interpolation on $\Phi_u$ gives
\[
 E_\theta\e_u=f_\theta(H)\e_u,\qquad
 E_\theta\e_v=f_\theta(H)\e_v,
 \qquad
 f_\theta(x)=\prod_{\eta\in\Phi_u\setminus\{\theta\}}
 \frac{x-\eta}{\theta-\eta}.
\]
Both vectors have entries in $L$. The spectral PST relation is
\[
 E_\theta\e_v=\beta_\theta E_\theta\e_u,
 \qquad \beta_\theta=\gamma^{-1}e^{-\ii\tau\theta}.
\]
Taking a nonzero coordinate shows that $\beta_\theta\in L$. Therefore
\[
 \frac{\beta_\theta}{\beta_\eta}=e^{-\ii\tau(\theta-\eta)}\in L.
\]
By~\eqref{eq:gaplattice}, an integer combination of the differences equals $\delta$. Multiplying the corresponding ratios, with their integer exponents, gives $z\in L$. This step cancels the unrestricted phase $\gamma$.

If $\delta$ is rational, then $L=K=J$. Otherwise let $\sigma$ be the automorphism of the biquadratic field $L=K(\sqrt\Delta)$ that fixes $K$ and negates $\sqrt\Delta$. It fixes $H$ and maps $f_\theta(H)\e_u$ to $f_{2c-\theta}(H)\e_u$, and similarly at $v$. Applying it to the preceding vector identity gives
\[
 \sigma(\beta_\theta)=\beta_{2c-\theta},\qquad
 \sigma\!\left(\frac{\beta_\theta}{\beta_\eta}\right)
 =\left(\frac{\beta_\theta}{\beta_\eta}\right)^{-1}.
\]
The same integer combination yields $\sigma(z)=z^{-1}$. Complex conjugation also sends $z$ to $z^{-1}$, since $|z|=1$. Thus $z$ is fixed by the product of complex conjugation and $\sigma$. This product negates both $\sqrt{-d}$ and $\sqrt\Delta$, so its fixed field is $\mathbb Q(\sqrt{-d\Delta})=J$.

For the final assertion, complex conjugation sends $H$ to $-H$, so $\Phi_u$ is symmetric about zero and $c=0$. The return phase $-1$ gives $\theta/\delta\in\frac12+\mathbb Z$ for every $\theta\in\Phi_u$. The group generated by $\Phi_u$ contains $\delta$, since $\delta$ is an integer combination of gaps. Subtracting an integer multiple of $\delta$ from any supported eigenvalue then gives $\delta/2$. All supported eigenvalues lie in $(\delta/2)\mathbb Z$, so the group is exactly $(\delta/2)\mathbb Z$. The transfer phase $\gamma$ is $\pm1$, so the relation $e^{-\ii\tau\theta}=\gamma\beta_\theta$ puts each of these exponentials in $L$. Taking the integer combination giving $\delta/2$ shows that $e^{-\ii\tau\delta/2}\in L$. If $\delta$ is irrational, the automorphism $\sigma$ sends each $\theta$ to $-\theta$ and each exponential to its inverse; complex conjugation does the same. The preceding fixed-field argument therefore puts $e^{-\ii\tau\delta/2}$ in $J$. The rational case follows from $L=J$.
\end{proof}

\section{Unweighted realization and target selection}\label{sec:selection}
The weighted matrices constructed later may have extra PST targets. We need to remove those targets and the edge weights at the same time. We use the equitable-partition method for state transfer~\cite{BFF}, with cell sizes chosen to control the target set. The construction replaces each vertex by an independent cell. A selected vertex becomes a singleton cell. Every other vertex becomes a larger cell, which spreads the incoming amplitude equally among its vertices. No individual vertex in that larger cell can then be a PST target of a singleton source.

\begin{lemma}\label{lem:cells}
Let $B$ be the Hermitian adjacency matrix of a connected integer-weighted mixed graph. Choose $r>0$ and positive integers $s_v$, at least one of which is one. Suppose that for every edge $uv$, of weight $w_{uv}=|B_{u,v}|$,
\begin{equation}\label{eq:cellconditions}
 d_{uv}=r w_{uv}\frac{s_v}{s_u}\in\mathbb Z_{>0},\qquad
 d_{vu}=r w_{uv}\frac{s_u}{s_v}\in\mathbb Z_{>0},\qquad
 r w_{uv}\leq s_us_v.
\end{equation}
There is a connected mixed graph $Y$ with independent cells $V_v$ of sizes $s_v^2$. Its Hermitian adjacency matrix satisfies
\[
 H_YQ=QrB,\qquad
 Q_{x,v}=\begin{cases}1/s_v,&x\in V_v,\\0,&x\notin V_v.\end{cases}
\]
If $V_p$ is a singleton, identify its vertex with $p$. Then, for $x\in V_v$,
\begin{equation}\label{eq:cellamplitude}
 U_Y(t)_{x,p}=\frac{U_B(rt)_{v,p}}{s_v}.
\end{equation}
If the input graph is oriented, then so is $Y$.
\end{lemma}
\begin{proof}
For each edge $uv$, construct a simple biregular bipartite graph between cells of sizes $m_u=s_u^2$ and $m_v=s_v^2$, with respective degrees $d_{uv},d_{vu}$. These degrees are feasible: $m_ud_{uv}=m_vd_{vu}$ and the last inequality in~\eqref{eq:cellconditions} gives $d_{uv}\leq m_v$ and $d_{vu}\leq m_u$. Explicitly, join row $0\leq i<m_u$ to columns
\[
 id_{uv}+j\pmod{m_v},\qquad 0\leq j<d_{uv}.
\]
Each row has distinct neighbors. The concatenated column sequence consists of $d_{vu}$ full residue lists, so all column degrees are $d_{vu}$. Give every edge in this block gain $B_{u,v}/w_{uv}$, with conjugate gains in the opposite block.

The columns of $Q$ are the normalized characteristic vectors of the cells, so $Q^*Q=I$. The row sums in each block give
\[
 (H_YQ)_{x,v}=\frac{B_{u,v}}{w_{uv}}\frac{d_{uv}}{s_v}
 =\frac{rB_{u,v}}{s_u}=(QrB)_{x,v}\qquad(x\in V_u).
\]
The intertwining identity gives $U_Y(t)Q=QU_B(rt)$. Since $Q\e_p=\e_p$ at a singleton cell, its $p$th column gives~\eqref{eq:cellamplitude}. A nonsingleton target has amplitude at most $1/s_v<1$, while a singleton target has modulus one exactly when the corresponding quotient target does. At $v=p$, the same formula shows that return times at the singleton source are exactly those for $B$, divided by $r$, and the return phases are unchanged.

A block incident with a singleton must be complete: its degree on the other side is a positive integer at most one. Thus all vertices in a neighboring cell are reachable from the singleton. Once a cell is reached in its entirety, positive degrees make every vertex in each neighboring cell reachable. Connectivity of the underlying graph of $B$ proves connectivity of $Y$.
\end{proof}

Equation~\eqref{eq:cellamplitude} gives the complete PST class of a singleton source $p$: its targets are precisely the singleton cells indexed by vertices in $\T_B(p)$. Transfer times between singleton vertices are the corresponding times for $B$, divided by $r$. The same rescaling applies to the minimum vertex period at $p$, whenever it exists.

The next theorem makes the cell-size conditions automatic. A Cartesian power makes the selected diagonal vertices independent and raises their transition entries to that power. In the oriented case, odd powers preserve the signs of all transfers and returns.

\begin{theorem}\label{thm:selection}
Let $H$ be the Hermitian adjacency matrix of a connected integer-weighted mixed graph on $n\geq2$ vertices. Let $S$ be a nonempty set of vertices and let $\ell\geq2$ be an integer. There exist an integer $r\geq2$ and a connected graph $Y$, with distinguished vertices $\widehat a$ for $a\in S$, such that
\begin{equation}\label{eq:terminal}
 U_Y(t)_{\widehat b,\widehat a}=U_H(rt)_{b,a}^{\,\ell}
 \qquad(a,b\in S,\ t\in\mathbb R),
\end{equation}
and
\begin{equation}\label{eq:leakage}
 |U_Y(t)_{x,\widehat a}|\leq\frac1r
 \qquad(a\in S,\ x\notin\{\widehat b:b\in S\},\ t\in\mathbb R).
\end{equation}
If the input graph is oriented, then so is $Y$.
\end{theorem}
\begin{proof}
Form the Cartesian-power Hamiltonian and selected diagonal set
\[
 B=\sum_{j=1}^{\ell} I_n^{\otimes(j-1)}\otimes H\otimes I_n^{\otimes(\ell-j)},
 \qquad \mathcal D=\{(a,\ldots,a):a\in S\}.
\]
The underlying graph of $B$ is connected, and $\mathcal D$ is independent. Every nonzero entry of $B$ comes from exactly one entry of $H$, so it retains its weight and allowed gain. For $v\notin\mathcal D$ adjacent to $p\in\mathcal D$, put $c_v=|B_{v,p}|$. This is well-defined. For $\ell\geq3$, a tuple cannot differ in just one coordinate from two distinct constant tuples. For $\ell=2$, the only possible shared neighbor is $(a,b)$ of $(a,a)$ and $(b,b)$, and both edge weights equal $|H_{a,b}|$. Put $c_v=1$ at the remaining vertices outside $\mathcal D$. Set
\[
 w_{\max}=\max_{a,b}|H_{a,b}|,\qquad
 L_c=\operatorname{lcm}\{c_v:v\notin\mathcal D\},
\]
and choose an integer $r\geq\max\{2,w_{\max}\}$ divisible by $L_c$.

Apply Lemma~\ref{lem:cells} with $s_p=1$ for $p\in\mathcal D$ and $s_v=rc_v$ elsewhere. An edge incident with $\mathcal D$, of weight $w$, has $s_v=rw$, so its block is complete. For an edge with both ends outside $\mathcal D$,
\[
 d_{uv}=rw\frac{c_v}{c_u},\qquad d_{vu}=rw\frac{c_u}{c_v}.
\]
These degrees are integers. Also $w\leq w_{\max}\leq r c_uc_v$, which gives the capacity inequality in~\eqref{eq:cellconditions}. There are no edges within $\mathcal D$, so the lemma applies.

The summands defining $B$ commute, and hence
\[
 U_B(t)=U_H(t)^{\otimes\ell}.
\]
This is the standard factorization for Cartesian-product walks; see~\cite[Eqs.~(9)--(10)]{CDEL}. For $v=(b_1,\ldots,b_\ell)$, $x\in V_v$, and $a\in S$, Lemma~\ref{lem:cells} gives
\[
 U_Y(t)_{x,\widehat a}
   =\frac{\prod_{j=1}^{\ell}U_H(rt)_{b_j,a}}{s_v}.
\]
This proves~\eqref{eq:terminal} at selected diagonal vertices and~\eqref{eq:leakage} elsewhere, since every other cell has $s_v\geq r$. An entry and its $\ell$th power have modulus one at exactly the same times. Consequently, transfer times between distinguished vertices and return times at distinguished vertices are divided by $r$, and a phase $\gamma$ becomes $\gamma^\ell$.

The eigenvalue supports satisfy
\begin{equation}\label{eq:liftsupport}
 \Phi_{\widehat a}(Y)
   =\{r(\theta_1+\cdots+\theta_\ell):
          \theta_1,\ldots,\theta_\ell\in\Phi_a(H)\}.
\end{equation}
To prove this formula, write $p_\theta=(E_\theta)_{a,a}>0$ for $\theta\in\Phi_a(H)$. The diagonal case of~\eqref{eq:terminal} is
\[
 U_Y(t)_{\widehat a,\widehat a}
   =\sum_{\theta_1,\ldots,\theta_\ell\in\Phi_a(H)}
      p_{\theta_1}\cdots p_{\theta_\ell}
      e^{-\ii rt(\theta_1+\cdots+\theta_\ell)}.
\]
All coefficients are positive, so equal frequencies add without cancellation. This proves~\eqref{eq:liftsupport}. Edge gains are preserved in the cell construction, proving the orientation assertion. For an oriented input, phases are $\pm1$, and an odd power preserves them.

The construction works for every integer $r\geq\max\{2,w_{\max}\}$ divisible by $L_c$, and its order is
\begin{equation}\label{eq:order}
 |V(Y)|=|S|+r^2\sum_{v\notin\mathcal D}c_v^2.
\end{equation}
\end{proof}

Equation~\eqref{eq:terminal} and bound~\eqref{eq:leakage} determine the complete target set:
\[
 \T_Y(\widehat a)=\{\widehat b:b\in S\cap\T_H(a)\}.
\]
All transfer times between distinguished vertices and all return times at distinguished vertices are exactly those for $H$, divided by $r$, and a phase $\gamma$ becomes $\gamma^\ell$. In particular, a minimum vertex period $T_a$ becomes $T_a/r$. For an oriented input, odd $\ell$ preserves every transfer and return phase. The additional spectral-support and graph-order formulas are~\eqref{eq:liftsupport} and~\eqref{eq:order}.

For a rational skew-symmetric input $A$, first clear its denominators by a positive integer $D$ and apply the theorem to $H=\ii DA$. The total time rescaling is $1/(Dr)$. If all input weights are one, then $c_v=1$ and~\eqref{eq:order} becomes $|S|+r^2(n^\ell-|S|)$. We use $\ell=2$ when no negative return phase must be retained, and $\ell=3$ when it must.

\section{Proof of the classification theorem}\label{sec:unbounded}
The arithmetic restriction provides the necessary field condition. We now prove sufficiency by constructing the prescribed transfers in any imaginary quadratic field. We first build a rational matrix and a rational unit vector whose orbit is orthonormal at the prescribed times. A rational change of coordinates then turns these orbit vectors into vertex states. Theorem~\ref{thm:selection} converts this intermediate model into an unweighted graph with exactly the required targets.

\subsection{A rational cosine polynomial}
For a real skew-symmetric matrix $A$ and a unit vector $\mathbf u$, the inner product of two orbit vectors satisfies
\[
 \langle e^{sA}\mathbf u,e^{tA}\mathbf u\rangle=\langle \mathbf u,e^{(t-s)A}\mathbf u\rangle.
\]
Thus a return amplitude that vanishes at the prescribed time differences will make those orbit vectors orthogonal. The following polynomial supplies these zeros. Its nonnegative coefficients will become spectral weights.

\begin{lemma}\label{lem:positive}
Let $\Lambda=\{x_0,\ldots,x_{k-1}\}\subset\mathbb R/\mathbb Z$ have $k$ elements, with $x_0=0$, and suppose $z_j=e^{2\pi\ii x_j}$ lies in $\mathbb Q(\sqrt{-\Delta})$ for every $j$, where $\Delta$ is a positive integer. There exists a cosine polynomial
\begin{equation}\label{eq:positivecos}
 F(t)=\sum_{m=0}^{M}b_m\cos(mt),\qquad
 b_m\in\mathbb Q_{\geq0},\quad b_0,b_1>0,\quad\sum_m b_m=1,
\end{equation}
satisfying $F(2\pi(x_j-x_\ell))=0$ whenever $j\ne\ell$.
\end{lemma}
\begin{proof}
For each unordered pair $j<\ell$, choose a positive integer $n_{j\ell}$ with
\[
 c_{j\ell}:=\operatorname{Re}\bigl((z_j\overline{z_\ell})^{n_{j\ell}}\bigr)<0.
\]
Such an integer always exists. Indeed, choose $\beta\in(0,\pi]$ with $\cos\beta=\operatorname{Re}(z_j\overline{z_\ell})$. If $\beta>\pi/2$, take $n=1$. Otherwise take $n=\lfloor\pi/(2\beta)\rfloor+1$, so $\pi/2<n\beta\leq\pi$. The real part of an element of $\mathbb Q(\sqrt{-\Delta})$ is rational, hence $c_{j\ell}\in\mathbb Q$.

Use one factor for each required zero, and define
\begin{equation}\label{eq:productkernel}
 F(t)=\frac{2+\cos t}{3}
 \prod_{0\leq j<\ell<k}
 \frac{\cos(n_{j\ell}t)-c_{j\ell}}{1-c_{j\ell}}.
\end{equation}
Every factor has nonnegative rational cosine coefficients and a positive constant coefficient. Products preserve this property, by
\[
 \cos(at)\cos(bt)=\tfrac12\cos((a+b)t)+\tfrac12\cos((a-b)t).
\]
The product has a positive constant coefficient, so multiplication by $(2+\cos t)/3$ gives both $b_0>0$ and $b_1>0$. The normalization gives $F(0)=1$. At each prescribed difference one factor vanishes. One may take $M=1+\sum_{j<\ell}n_{j\ell}$. For $k=1$ the product is empty, and the same formula applies.
\end{proof}
The coefficients of $F$ are nonnegative; the values of $F$ need not be. Each integer $n_{j\ell}$ can be found by enumerating powers until the rational real part becomes negative.

\subsection{Rational orthogonal realization}
We realize each frequency of $F$ by a rotation block. The construction uses Lagrange's four-square theorem.

\begin{lemma}[{\cite[Theorem~369]{HW}}]\label{lem:foursquares}
Every positive integer is a sum of four squares of integers.
\end{lemma}

The four-square representation of $\Delta$ gives an integer rotation block with the required angular speed.

\begin{lemma}\label{lem:rotation}
For every positive integer $\Delta$ there is an integer skew-symmetric matrix $\Omega_\Delta$ of order eight such that $\Omega_\Delta^2=-\Delta I$. For $\Delta=1$, a matrix of order two suffices.
\end{lemma}
\begin{proof}
By Lemma~\ref{lem:foursquares}, write $\Delta=a_0^2+a_1^2+a_2^2+a_3^2$, and put
\[
 M_\Delta=\begin{pmatrix}
 a_0&-a_1&-a_2&-a_3\\
 a_1&a_0&-a_3&a_2\\
 a_2&a_3&a_0&-a_1\\
 a_3&-a_2&a_1&a_0
 \end{pmatrix},\qquad
 \Omega_\Delta=\begin{pmatrix}O_4&-M_\Delta\\M_\Delta^{\mathsf T}&O_4\end{pmatrix}.
\]
Direct multiplication gives $M_\Delta M_\Delta^{\mathsf T}=M_\Delta^{\mathsf T}M_\Delta=\Delta I_4$, proving the assertion. For $\Delta=1$ use $\Omega_1=\left(\begin{smallmatrix}0&-1\\1&0\end{smallmatrix}\right)$.
\end{proof}

\begin{lemma}\label{lem:frame}
Let $\Delta$ be a positive integer and let $\alpha_0=0,\alpha_1,\ldots,\alpha_{k-1}\in\mathbb R$ satisfy $z_j=e^{\ii \alpha_j}\in\mathbb Q(\sqrt{-\Delta})$. Suppose
\[
 F(t)=\sum_{m=0}^{M}b_m\cos(mt),\qquad
 b_m\in\mathbb Q_{\geq0},\quad b_1>0,\quad\sum_m b_m=1,
\]
and $F(\alpha_j-\alpha_h)=0$ for $j\ne h$. There are an integer skew-symmetric matrix $A_0$ and a rational unit vector $\mathbf u$ such that
\[
 \langle\mathbf u,e^{tA_0}\mathbf u\rangle=F(\sqrt\Delta t),
\]
and the vectors $\mathbf u_j=e^{(\alpha_j/\sqrt\Delta)A_0}\mathbf u$ form a rational orthonormal frame. The eigenvalue support of $\mathbf u$ for $\ii A_0$ is
\begin{equation}\label{eq:framesupport}
 \{\pm m\sqrt\Delta:0\leq m\leq M,\ b_m>0\}.
\end{equation}
If $b_m=0$ for every even $m$, then $\mathbf u$ has minimum period $\pi/\sqrt\Delta$ and return phase $-1$. Otherwise, $\mathbf u$ has minimum period $2\pi/\sqrt\Delta$ and return phase $+1$.
\end{lemma}
\begin{proof}
Every nonnegative rational number is the squared norm of a rational vector in $\mathbb Q^4$. Indeed, if $b=p/q>0$ with $p,q\in\mathbb Z_{>0}$, Lemma~\ref{lem:foursquares} gives $pq=a_1^2+\cdots+a_4^2$, so $b=\sum_j(a_j/q)^2$. Choose $\mathbf v_m\in\mathbb Q^4$ with $\|\mathbf v_m\|^2=b_m$.

Let $\Omega$ have order $d_\Delta$, where $d_\Delta=8$ in general and $d_1=4$, and satisfy $\Omega^{\mathsf T}=-\Omega$ and $\Omega^2=-\Delta I$. Use Lemma~\ref{lem:rotation} for the order-eight choice and $\Omega_1\oplus\Omega_1$ for $\Delta=1$. For $m\geq1$ with $b_m>0$, extend $\mathbf v_m$ by zeros to $\widetilde{\mathbf v}_m\in\mathbb Q^{d_\Delta}$, and set
\[
 A_0=O_4\oplus\bigoplus_{\substack{1\leq m\leq M\\b_m>0}}m\Omega,
 \qquad
 \mathbf u=\mathbf v_0\oplus\bigoplus_{\substack{1\leq m\leq M\\b_m>0}}\widetilde{\mathbf v}_m.
\]
Then $\|\mathbf u\|=1$. The identity
\begin{equation}\label{eq:rotationexp}
 e^{tm\Omega}=\cos(m\sqrt\Delta t)I+
 \frac{\sin(m\sqrt\Delta t)}{\sqrt\Delta}\Omega
\end{equation}
gives the return amplitude and support. The spectral weight is $b_0$ at zero, when present, and $b_m/2$ at each of $\pm m\sqrt\Delta$ for $m>0$. Write $z_j^m=p_{jm}+q_{jm}\sqrt{-\Delta}$ with $p_{jm},q_{jm}\in\mathbb Q$. At time $\alpha_j/\sqrt\Delta$, the block in~\eqref{eq:rotationexp} is $p_{jm}I+q_{jm}\Omega$. Thus every $\mathbf u_j$ is rational, and
\[
 \langle\mathbf u_j,\mathbf u_h\rangle=F(\alpha_h-\alpha_j)=\delta_{jh},
\]
where $\delta_{jh}$ is the Kronecker delta.

The positive weight at frequency one puts both $\sqrt\Delta$ and $-\sqrt\Delta$ in the support. Any return therefore satisfies $\sqrt\Delta t\in\pi\mathbb Z$. At $t=\pi/\sqrt\Delta$, all supported phases agree exactly when only odd frequencies occur, in which case they are all $-1$. Otherwise the first return is at $2\pi/\sqrt\Delta$, with phase $+1$. The block orders give a matrix order at most $4+8M$, or $4+4M$ when $\Delta=1$.
\end{proof}

The vectors $\mathbf u_j$ are rational and orthonormal, but they need not be coordinate vectors. A rational orthogonal change of coordinates resolves this last issue.

\begin{lemma}\label{lem:completion}
Every rational orthonormal frame $\mathbf u_0,\ldots,\mathbf u_{k-1}\in\mathbb Q^N$ can be sent to $\e_0,\ldots,\e_{k-1}$ by a rational orthogonal matrix $R$.
\end{lemma}
\begin{proof}
Proceed one vector at a time. After fixing the earlier images, let $\mathbf y$ be the current image of $\mathbf u_j$. If $\mathbf y=\e_j$, do nothing. Otherwise the Householder reflection~\cite{Householder}
\[
 I-\frac{2\mathbf v\mathbf v^{\mathsf T}}{\mathbf v^{\mathsf T}\mathbf v},\qquad \mathbf v=\mathbf y-\e_j,
\]
sends $\mathbf y$ to $\e_j$. Its entries are rational because $\mathbf v$ is rational. Since $\mathbf y$ is orthogonal to the earlier coordinate vectors, the first $j$ entries of $\mathbf v$ vanish, so those vectors remain fixed. The product of these reflections proves the assertion.
\end{proof}

\subsection{Completion of the proof}\label{sec:phaseconstruction}
\begin{proof}[Proof of Theorem~\ref{thm:schedules}]
Write $\Lambda=\{x_0,\ldots,x_{k-1}\}$ with $x_0=0$, and put $z_j=e^{2\pi\ii x_j}$.
We first prove \textup{(iii)} $\Longleftrightarrow$ \textup{(iv)}. If the unit-circle points lie in one imaginary quadratic field, so do their pairwise ratios. Taking real parts proves \textup{(iv)}. Conversely, write $z_j=a_j+\ii b_j$. Condition \textup{(iv)} with $x_0=0$ gives $a_j\in\mathbb Q$ and $b_j^2=1-a_j^2\in\mathbb Q$. If every $b_j$ vanishes, all points are $\pm1$ and lie in $\mathbb Q(\ii)$. Otherwise choose $b_h\ne0$. Since
\[
 a_ja_h+b_jb_h=\cos(2\pi(x_j-x_h))\in\mathbb Q,
\]
we have $b_j/b_h\in\mathbb Q$ for every $j$. All $z_j$ then lie in the imaginary quadratic field $\mathbb Q(\ii b_h)$.

Next, \textup{(ii)} $\Longrightarrow$ \textup{(i)} is immediate because an oriented graph is a mixed graph. If \textup{(i)} holds, its Hermitian adjacency matrix has entries in $\mathbb Q(\ii)$. Lemma~\ref{lem:field} places every $e^{-2\pi\ii x_j}$ in one imaginary quadratic field. This field is closed under complex conjugation, so \textup{(iii)} follows.

For an oriented realization with return phase $-1$, Lemma~\ref{lem:field} also places the half-angle points $e^{\pi\ii x_j}$ in a common imaginary quadratic field. Taking real parts of their pairwise ratios proves~\eqref{eq:negativephasecondition}. Conversely, the argument for \textup{(iii)} $\Longleftrightarrow$ \textup{(iv)}, applied to these half-angle points, shows that~\eqref{eq:negativephasecondition} puts them in one imaginary quadratic field.

We now prove sufficiency for both return phases in one construction. Fix $\varepsilon\in\{1,-1\}$, and put $\nu=1$ if $\varepsilon=1$ and $\nu=2$ if $\varepsilon=-1$. Suppose that
\[
 e^{2\pi\ii x_j/\nu}\in\mathbb Q(\sqrt{-\Delta})
 \qquad(0\leq j<k),
\]
where $\Delta$ is positive and squarefree. Their $\nu$th powers are $z_j$, so Lemma~\ref{lem:positive} gives a polynomial $G$ for $\Lambda$. Set
\[
 F_\nu(t)=
 \begin{cases}
 G(t),&\nu=1,\\
 \cos t\,G(2t),&\nu=2,
 \end{cases}
 \qquad \alpha_j=\frac{2\pi x_j}{\nu}.
\]
In both cases $F_\nu(0)=1$ and $F_\nu(\alpha_j-\alpha_h)=0$ for $j\ne h$. For $\nu=1$, the coefficients at frequencies zero and one are positive. For $\nu=2$, the cosine product formula gives nonnegative rational coefficients, only odd frequencies, and a positive coefficient at frequency one. Lemma~\ref{lem:frame} therefore gives a rational orthonormal orbit frame whose initial vector has minimum period
\begin{equation}\label{eq:modelperiod}
 P=\frac{2\pi}{\nu\sqrt\Delta}
\end{equation}
and return phase $\varepsilon$. Lemma~\ref{lem:completion} gives a rational skew-symmetric matrix $A_{\mathrm{rat}}$ satisfying
\[
 e^{Px_jA_{\mathrm{rat}}}\e_0=\e_j\qquad(0\leq j<k).
\]
Restrict $A_{\mathrm{rat}}$ to the connected component containing $0$. All $j<k$ belong to this component, since they are reached from $0$. The source is nonstationary, so the component has at least two vertices. Restriction preserves its source orbit and minimum vertex period $P$.

Choose a positive integer $D$ clearing the denominators of the restricted matrix. Apply Theorem~\ref{thm:selection} to $\ii D A_{\mathrm{rat}}$, with selected set $\{0,\ldots,k-1\}$ and $\ell=\nu+1$. Equations~\eqref{eq:terminal} and~\eqref{eq:leakage} make the selected vertices the complete PST class of the distinguished source. Their common minimum vertex period is $T=P/(Dr)$, their normalized transfer times are $x_j$, and each displayed source-to-target transfer has phase $+1$. The return phase is $\varepsilon^{\nu+1}=\varepsilon$. For $\nu=1$ this proves \textup{(iii)} $\Longrightarrow$ \textup{(ii)}; for $\nu=2$ it proves sufficiency of~\eqref{eq:negativephasecondition}.

Finally, prescribe signs $\sigma_x\in\{1,-1\}$ with $\sigma_0=1$, and index the class of the realizing graph $X$ as $\{v_x:x\in\Lambda\}$. Conjugate its skew adjacency matrix by the diagonal sign matrix with entry $\sigma_x$ at $v_x$ and entry one elsewhere. The resulting graph remains unweighted and oriented. This changes the source-to-target phase at $Tx$ to $\sigma_x$ without changing any transfer time or return phase. Write $U_X$ for its transition matrix. Lemma~\ref{lem:labels}, with representatives $x,y\in[0,1)$, gives all transfers and their phases:
\begin{equation}\label{eq:allphases}
 U_X\bigl(T(y-x+n)\bigr)\e_{v_x}
   =\varepsilon^n\sigma_y\sigma_x\e_{v_y}
 \qquad(x,y\in\Lambda,\ n\in\mathbb Z).
\end{equation}
\end{proof}

Every step is finite and uses exact arithmetic when the unit-circle points are given in a common imaginary quadratic field. One finds the powers in~\eqref{eq:productkernel}, expands the polynomial, chooses four-square decompositions, and forms the rational reflections. Clearing denominators and applying Theorem~\ref{thm:selection} completes the construction. The resulting graph can be large because the denominators can be large. There is no need to classify extra targets of the intermediate matrix: the selection theorem removes them.

\section{Consequences of the classification}\label{sec:consequences}
We first describe rational transfer times and the partitions they define. We then relate coordinate permutations to switching automorphisms, construct classes with irrational normalized transfer times, and treat simultaneous realization with prescribed relative minimum vertex periods.

\subsection{Rational normalized transfer times}
A rational normalized transfer time corresponds to a root of unity in the field supplied by Lemma~\ref{lem:field}. We begin with the resulting bound, then determine all possible sets and partitions.

For an imaginary quadratic field $J$, write $\mu(J)$ for its group of roots of unity and put $w=|\mu(J)|$. Then
\[
 w=\begin{cases}
 4,&J=\mathbb Q(\ii),\\
 6,&J=\mathbb Q(\sqrt{-3}),\\
 2,&\text{otherwise}.
 \end{cases}
\]
Indeed, a primitive $m$th root of unity generates a field of degree $\varphi(m)$ over $\mathbb Q$, where $\varphi$ is Euler's totient function~\cite[Lemma~5.9 and Theorem~5.10]{Milne}. In a quadratic field, $\varphi(m)\leq2$, so $m\in\{1,2,3,4,6\}$. Orders four and three generate $\mathbb Q(\ii)$ and $\mathbb Q(\sqrt{-3})$, respectively.

\begin{theorem}\label{thm:rational}
Let $H$ be Hermitian over an imaginary quadratic field, and let $u$ have minimum vertex period $T$. Let $J$ be the field in~\eqref{eq:targetfield} and put $w=|\mu(J)|$. Every target reached at a rational multiple of $T$ has normalized transfer time
\[
 x_v\in\tfrac1w\mathbb Z/\mathbb Z.
\]
There are at most $w$ such targets, including $u$. If $H=\ii A$ with $A$ rational and real skew-symmetric and return phase $-1$ at $T$, the grid and bound improve to $\frac2w\mathbb Z/\mathbb Z$ and $w/2$, respectively.
\end{theorem}
\begin{proof}
If $x_v$ is rational, then $e^{-2\pi\ii x_v}$ is a root of unity in $J$ by Lemma~\ref{lem:field}, proving the grid restriction. Distinct targets have distinct normalized transfer times by Lemma~\ref{lem:labels}, giving the bound. For the negative return phase, \eqref{eq:halftimefield} puts $e^{-\pi\ii x_v}$ in $\mu(J)$. Hence $wx_v\in2\mathbb Z$ for a representative $x_v\in[0,1)$, giving a grid of size $w/2$.
\end{proof}

For Hermitian matrices over $\mathbb Q(\ii)$, and in particular for mixed graphs, the sharper field-dependent conclusion is
\[
\renewcommand{\arraystretch}{1.3}
\begin{array}{c|c|c}
\text{Gap generator}&\substack{\text{Rational normalized}\\\text{transfer times}}&\text{Target bound}\\\hline
\delta\in\mathbb Q&\frac14\mathbb Z/\mathbb Z&4\\
\delta\in\mathbb Q_{>0}\sqrt3&\frac16\mathbb Z/\mathbb Z&6\\
\delta\in\mathbb Q_{>0}\sqrt\Delta,\ \Delta>1,\ \Delta\ne3
&\frac12\mathbb Z/\mathbb Z&2.
\end{array}
\]
The quadratic field $J$ is needed for this restriction. Merely placing the point $e^{-\ii\delta\tau}$ in the degree-four splitting field would not exclude eighth or twelfth roots of unity. Lemma~\ref{lem:field} excludes them by using both field conjugations.

Suppose PST occurs between distinct vertices at a rational multiple of their common minimum vertex period $T$. The earlier of the first transfer times in the two directions, divided by $T$, is consequently one of $1/6,1/4,1/3,1/2$. Corollary~\ref{cor:pairs} below shows that each value occurs even for a complete PST class of size two in an oriented graph.

\begin{corollary}\label{cor:pairs}
Let $0<\rho\leq1/2$. There exists a connected oriented graph with a complete PST class $\{a,b\}$ and common minimum vertex period $T$, whose first transfer times from $a$ to $b$ and from $b$ to $a$ are $\rho T$ and $(1-\rho)T$, respectively, if and only if $\cos(2\pi\rho)\in\mathbb Q$.
\end{corollary}
\begin{proof}
Apply Theorem~\ref{thm:schedules} to $\{0,\rho\}$. The only nontrivial cosine condition is the one displayed. Lemma~\ref{lem:labels} gives the first transfer times $\rho T$ and $(1-\rho)T$.
\end{proof}
Since $\rho\leq1/2$, the earlier of the two times is $\rho T$. PST occurs in both directions at that time if and only if $\rho=1/2$.

\begin{corollary}\label{cor:jointgrids}
A finite set $\Lambda\subset\mathbb Q/\mathbb Z$ containing zero is the set of normalized transfer times of a complete PST class at a periodic vertex in a connected graph if and only if
\[
 \Lambda\subseteq\tfrac14\mathbb Z/\mathbb Z
 \quad\text{or}\quad
 \Lambda\subseteq\tfrac16\mathbb Z/\mathbb Z.
\]
Every such set can already be realized in a connected oriented graph with return phase $+1$. An oriented realization with return phase $-1$ exists if and only if
\[
 \Lambda\subseteq\tfrac12\mathbb Z/\mathbb Z
 \quad\text{or}\quad
 \Lambda\subseteq\tfrac13\mathbb Z/\mathbb Z.
\]
\end{corollary}
\begin{proof}
In a common imaginary quadratic field, all the corresponding unit-circle points are roots of unity. The group generated by these points is contained either in the fourth roots or in the sixth roots, by the quadratic root-of-unity classification used in Theorem~\ref{thm:rational}. Conversely, if $\Lambda$ lies in either displayed grid, the points $e^{2\pi\ii x}$, $x\in\Lambda$, lie in $\mathbb Q(\ii)$ or $\mathbb Q(\sqrt{-3})$, respectively. Theorem~\ref{thm:schedules} therefore realizes $\Lambda$ with return phase $+1$.

For return phase $-1$, the points $e^{\pi\ii x}$ must lie in one imaginary quadratic field. When $x$ is rational, these are roots of unity of order dividing four or six. Consequently $x$ lies in the half-period grid or the third-period grid, respectively. Conversely, either grid puts the half-angle points in $\mathbb Q(\ii)$ or $\mathbb Q(\sqrt{-3})$, and the negative-phase part of Theorem~\ref{thm:schedules} applies.
\end{proof}
For example, $\{0,1/4,1/3\}$ cannot occur as a set of normalized transfer times, although each of $1/4$ and $1/3$ is individually realizable for a pair. The condition on a common field is a joint restriction.

Every size $1,2,3,4,5,6$ therefore occurs with all normalized transfer times rational. Corollary~\ref{cor:pairs} also shows that all four possible rational values of the earlier normalized transfer time are attained by oriented graphs with a complete class of size two.

For a complete PST class $C$ with minimum vertex period $T$, define
\[
 b\sim_{\mathbb Q}c
 \quad\Longleftrightarrow\quad
 \text{PST occurs from $b$ to $c$ at some }t\text{ with }t/T\in\mathbb Q.
\]
By~\eqref{eq:labeldifference}, this is equivalent to $x_c-x_b\in\mathbb Q/\mathbb Z$. In particular, $\sim_{\mathbb Q}$ is an equivalence relation, independent of the choice of source used to define the normalized transfer times.

Let $J$ be the field supplied by Lemma~\ref{lem:field} and put $z_b=e^{2\pi\ii x_b}$. The difference $x_c-x_b$ is rational modulo one exactly when $z_c/z_b\in\mu(J)$. Thus the equivalence classes correspond to the nonempty intersections
\[
 \{z_b:b\in C\}\cap z\mu(J),\qquad z\in J,\quad |z|=1.
\]
The following theorem classifies the finite partitions that can occur.

\begin{theorem}\label{thm:rationalblocks}
A partition of a finite nonempty set is realizable as the partition of a complete PST class under $\sim_{\mathbb Q}$, at periodic vertices in a connected graph, if and only if every part has size at most six. Every such partition has a connected oriented realization with return phase $+1$ at the minimum vertex period.

A connected oriented realization with return phase $-1$ at the minimum vertex period exists if and only if every part has size at most three.
\end{theorem}
\begin{proof}
The coset description bounds each part by $|\mu(J)|\leq6$. For an oriented graph with return phase $-1$, apply Theorem~\ref{thm:rational} from a reference vertex in each part; the bound becomes $|\mu(J)|/2\leq3$.

For sufficiency, we first give a construction in any imaginary quadratic field $J$. Fix $m\geq1$ and nonempty subsets $\mathcal R_0,\ldots,\mathcal R_{m-1}\subseteq\mu(J)$, each containing one. Choose a unit-modulus element $\zeta\in J$ of infinite multiplicative order. Such elements exist because, for $J=\mathbb Q(\sqrt{-\Delta})$, the points
\[
 \frac{1-\Delta q^2}{1+\Delta q^2}
 +\frac{2q}{1+\Delta q^2}\sqrt{-\Delta},\qquad q\in\mathbb Q,
\]
are infinitely many distinct norm-one elements, whereas $\mu(J)$ is finite. These points are also dense in the unit circle: for real $q$ the parametrization covers the circle except $-1$, and $\mathbb Q$ is dense in $\mathbb R$. Set
\[
 \mathcal Z=\bigcup_{j=0}^{m-1}\zeta^j\mathcal R_j.
\]
Apply Theorem~\ref{thm:schedules} to $\Lambda=\{x\in\mathbb R/\mathbb Z:e^{2\pi\ii x}\in\mathcal Z\}$. Since $1\in\mathcal Z$, this set contains zero. The subsets $\zeta^j\mathcal R_j$ lie in distinct cosets of $\mu(J)$: if $\zeta^j\mu(J)=\zeta^\ell\mu(J)$ for $j\ne\ell$, then $\zeta^{j-\ell}$ is a root of unity, contrary to the choice of $\zeta$. Thus they are precisely the equivalence classes of $\sim_{\mathbb Q}$, with the prescribed ratios of unit-circle points.

Given any partition with parts of sizes at most six, take $J=\mathbb Q(\sqrt{-3})$ and choose $\mathcal R_j$ of the corresponding sizes in its six-element root-of-unity group. Theorem~\ref{thm:schedules} gives the return phase $+1$.

For a partition with parts of sizes at most three and prescribed return phase $-1$, take $J=\mathbb Q(\sqrt{-3})$, choose the same infinite-order $\zeta$, and choose $\mathcal R_j$ of the prescribed sizes from the three cube roots of unity, each containing one. Set $\mathcal Z=\bigcup_j\zeta^j\mathcal R_j$ and define $\Lambda$ by
\[
 \{e^{2\pi\ii x}:x\in\Lambda\}=\{z^2:z\in\mathcal Z\}.
\]
Squaring is injective on $\mathcal Z$: within each $\zeta^j\mathcal R_j$, this follows because squaring is injective on the group of cube roots of unity; between different subsets, it follows from the infinite order of $\zeta$. Every half-angle point of $\Lambda$ is $z$ or $-z$ for some $z\in\mathcal Z$, so it lies in $J$. A quotient is a root of unity if and only if its square is, so squaring preserves the prescribed equivalence classes. The negative-phase realization in Theorem~\ref{thm:schedules} therefore has exactly the prescribed partition.
\end{proof}

The construction gives a field-dependent refinement. Fix an imaginary quadratic field $J$, an integer $m\geq1$, and nonempty subsets $\mathcal R_0,\ldots,\mathcal R_{m-1}\subseteq\mu(J)$ containing one. There is a connected oriented realization with equivalence classes $C_0,\ldots,C_{m-1}$ such that, for suitable reference vertices $b_j\in C_j$,
\[
 \{z_b/z_{b_j}:b\in C_j\}=\mathcal R_j.
\]
Here $z_b=e^{2\pi\ii x_b}$, with normalized transfer times relative to a fixed source in the complete PST class. This realization has return phase $+1$.

The theorem prescribes all equivalence classes of $\sim_{\mathbb Q}$ within one complete PST class. In particular, for every $N\geq6$, a partition with one part of size six and $N-6$ singleton parts gives a complete PST class of size $N$ with exactly six vertices reached at rational multiples of the period from a source in the first part. More generally, any subset of $\mu(J)$ containing one can be the set of unit-circle points corresponding to those targets, while arbitrarily many other vertices are reached only at irrational multiples of the period.

\subsection{Coordinate permutations and switching automorphisms}\label{sec:switching}
Pairwise transfer need not occur at one common time. Let $C=\T_H(a)$ have common minimum vertex period $T$, and write $\Lambda=\{x_b:b\in C\}$ as in Lemma~\ref{lem:labels}. To determine when one transition matrix permutes all coordinate lines of $C$, define the \emph{translation stabilizer} of $\Lambda$ by
\[
 \operatorname{Stab}(\Lambda)=\{s\in\R/\mathbb Z:\Lambda+s=\Lambda\}.
\]
A unitary monomial matrix has exactly one nonzero entry in each row and column. If it commutes with $H$, we call it a \emph{switching automorphism}. This extends the signed permutation definition for oriented graphs in~\cite{GL}.

\begin{lemma}\label{lem:stabilizer}
Let $C=\T_H(a)$, where $a$ has minimum vertex period $T$, and let $\Lambda=\{x_b:b\in C\}$ be its set of normalized transfer times relative to $a$. Then $\operatorname{Stab}(\Lambda)$ is a finite cyclic group, of order $h$ dividing $|C|$, and
\[
 \operatorname{Stab}(\Lambda)=\tfrac1h\mathbb Z/\mathbb Z.
\]
The matrix $U_H(t)$ maps every coordinate line of $C$ to a coordinate line of $C$ if and only if $t/T\pmod1\in \operatorname{Stab}(\Lambda)$. Its action is translation of the normalized transfer times by $t/T$.

If a switching automorphism maps a vertex of $C$ into $C$, it preserves $C$ and induces a translation in $\operatorname{Stab}(\Lambda)$.
\end{lemma}
\begin{proof}
For any fixed $x\in\Lambda$, the map $s\mapsto x+s$ embeds $\operatorname{Stab}(\Lambda)$ into $\Lambda$, so the stabilizer is finite. A finite subgroup of $\R/\mathbb Z$ of order $h$ is $\frac1h\mathbb Z/\mathbb Z$. Translation acts freely on $\Lambda$, and hence $h$ divides $|\Lambda|=|C|$. The assertion about $U_H(t)$ follows directly from~\eqref{eq:labeldifference}.

A commuting monomial matrix $\mathcal M$ sends complete PST classes to complete PST classes. If its image of $C$ intersects $C$, the two classes are equal. Write $U_H(Tx_b)\e_a=\gamma_b\e_b$ and $\mathcal M\e_b=\eta_b\e_{\pi(b)}$ for $b\in C$, where $\pi$ is the induced permutation, and put $d=\pi(a)$. Commutation with $U_H(Tx_b)$ gives
\[
 U_H(Tx_b)\e_d=\gamma_b\eta_b\eta_a^{-1}\e_{\pi(b)}.
\]
Lemma~\ref{lem:labels} therefore gives $x_{\pi(b)}=x_b+x_d$ in $\R/\mathbb Z$. Thus $\pi$ acts by translation and $x_d\in\operatorname{Stab}(\Lambda)$. Every element of $\operatorname{Stab}(\Lambda)$ is rational modulo one.
\end{proof}

If all differences between distinct normalized transfer times are irrational modulo one, the translation stabilizer is trivial. Lemma~\ref{lem:stabilizer} then excludes every switching automorphism sending one class vertex to a distinct class vertex.

\begin{proposition}\label{prop:closure}
Let $C=\T_H(a)$, where $a$ has minimum vertex period $T$, and write $\Lambda=\{x_b:b\in C\}$ as in Lemma~\ref{lem:labels}. Suppose $|C|=k\geq2$, and let $\tau$ be the first positive time at which PST occurs from $a$ to a distinct vertex. The following are equivalent:
\begin{enumerate}[label=\textup{(\roman*)},leftmargin=*]
\item $U_H(\tau)$ maps every coordinate line of $C$ to a coordinate line of $C$;
\item $\Lambda=\frac1k\mathbb Z/\mathbb Z$;
\item $|\operatorname{Stab}(\Lambda)|=k$.
\end{enumerate}
When they hold, $\tau=T/k$ and the induced permutation of $C$ is a single $k$-cycle.
\end{proposition}
\begin{proof}
Assume \textup{(i)}, and let $h=|\operatorname{Stab}(\Lambda)|$. Then $s=\tau/T$ belongs to $\operatorname{Stab}(\Lambda)\setminus\{0\}$. Since $1/h\in \operatorname{Stab}(\Lambda)\subseteq\Lambda$, minimality gives $s=1/h$. Every orbit of $\operatorname{Stab}(\Lambda)$ on $\Lambda$ has a representative in $[0,1/h)$. Minimality allows only the representative zero. Thus $\Lambda=\operatorname{Stab}(\Lambda)$ and $h=k$, proving \textup{(ii)}. The equivalence of \textup{(ii)} and \textup{(iii)} follows from $0\in\Lambda$ and the description of the stabilizer in Lemma~\ref{lem:stabilizer}. Finally, \textup{(ii)} gives $\tau=T/k$ and the stated cyclic action.
\end{proof}
The five-vertex class in Section~\ref{sec:six} has normalized transfer times $0,1/6,2/6,3/6,4/6$. Its stabilizer is trivial, and the transition at $T/6$ does not permute the five coordinate lines. Thus pairwise PST alone does not imply the cyclic conclusion stated in~\cite[Lemma 10 and Theorem 11]{SL}; condition \textup{(i)} of Proposition~\ref{prop:closure} is needed. Under the Bose--Mesner hypothesis of~\cite{CPRS}, that condition follows from~\cite[Lemma 2.1]{CPRS}.

\begin{corollary}\label{cor:symmetrybound}
Let $C=\T_H(a)$ for a Hermitian matrix over an imaginary quadratic field, where $a$ has minimum vertex period $T$. Let $\Lambda$ be the set of normalized transfer times, $h=|\operatorname{Stab}(\Lambda)|$, and $w=|\mu(J)|$ for the field in~\eqref{eq:targetfield}. Then
\[
 h\mid w,\qquad h\mid |C|.
\]
The permutations induced on $C$ by switching automorphisms preserving $C$ form a cyclic group of order dividing $h$. For an oriented graph with return phase $-1$ at $T$, also $h\mid w/2$.
\end{corollary}
\begin{proof}
By Lemma~\ref{lem:stabilizer}, $\operatorname{Stab}(\Lambda)=\frac1h\mathbb Z/\mathbb Z$ and $h\mid|C|$. Since $0\in\Lambda$, every element of $\operatorname{Stab}(\Lambda)$ is a rational normalized transfer time. Theorem~\ref{thm:rational} gives
\[
 \operatorname{Stab}(\Lambda)\subseteq\tfrac1w\mathbb Z/\mathbb Z,
\]
so $h\mid w$. Translation by $1/h$ has orbits of size $h$. The induced switching permutations form a subgroup of this cyclic translation group and also act freely. In the oriented case with return phase $-1$, the sharper grid in Theorem~\ref{thm:rational} gives $h\mid w/2$.
\end{proof}

The translation by $1/h$ shows that $U_H(T/h)$ acts on the coordinate lines of $C$ as a disjoint union of $h$-cycles. A transitive induced switching group therefore forces $|C|\in\{1,2,3,4,6\}$, or $|C|\in\{1,2,3\}$ in the oriented case with return phase $-1$. If $\gcd(|C|,6)=1$, the divisibility conditions give $h=1$, so every switching automorphism preserving $C$ fixes each of its vertices.

\subsection{Classes with irrational normalized transfer times}
Taking every part in Theorem~\ref{thm:rationalblocks} to be a singleton gives classes with no rational transfer times between distinct vertices. To control the size of the eigenvalue support as well, we use the following formula for the Fej\'er kernel.

\begin{lemma}[{\cite[Chapter~2, Lemma~5.1, p.~53]{SS}}]\label{lem:fejer}
For an integer $k\geq1$, the Fej\'er kernel
\[
 \mathcal F_k(t)=\frac1k\sum_{j=0}^{k-1}\sum_{m=-j}^{j}e^{\ii mt}
 \qquad(t\in\mathbb R)
\]
satisfies
\[
 \mathcal F_k(t)=\frac{\sin^2(kt/2)}{k\sin^2(t/2)}
 \qquad(t\notin2\pi\mathbb Z),
\]
with continuous extension $\mathcal F_k(2\pi r)=k$ for $r\in\mathbb Z$.
\end{lemma}

For $k\geq2$, put $K_k(t)=\mathcal F_k(t)/k$. Expanding the finite sum gives
\begin{equation}\label{eq:normalizedfejer}
 K_k(t)=\frac1k+\frac2{k^2}\sum_{m=1}^{k-1}(k-m)\cos(mt).
\end{equation}
Lemma~\ref{lem:fejer} gives $K_k(0)=1$ and $K_k(2\pi j/k)=0$ for $1\leq j<k$. The cosine coefficients in~\eqref{eq:normalizedfejer} are strictly positive in every degree from zero to $k-1$.

\begin{corollary}\label{thm:main}
For every integer $k\geq2$ there is a connected oriented graph with a complete PST class $C$ of size $k$ and minimum vertex period $T$ such that every transfer between distinct vertices of $C$ occurs at an irrational multiple of $T$. No switching automorphism sends one vertex of $C$ to a distinct vertex of $C$. The common eigenvalue support of these vertices can be chosen to have exactly $4k-3$ elements.
\end{corollary}
\begin{proof}
Put $\alpha_*=2\pi/k$ and define, for $\alpha$ near $\alpha_*$,
\begin{equation}\label{eq:fejerperturb}
 F_\alpha(t)=\prod_{j=1}^{k-1}
 \frac{\cos t-\cos(j\alpha)}{1-\cos(j\alpha)}.
\end{equation}
At $\alpha=\alpha_*$, the polynomial equals the normalized kernel $K_k$ in~\eqref{eq:normalizedfejer}. Indeed, as polynomials in $\cos t$, the product in~\eqref{eq:fejerperturb} and $K_k$ have the same degree and roots with multiplicities, and both equal one at $t=0$. All denominators in~\eqref{eq:fejerperturb} are nonzero at $\alpha_*$, and the coefficients in~\eqref{eq:normalizedfejer} are strictly positive. By continuity, both properties persist in a sufficiently small neighborhood of $\alpha_*$. Choose $\alpha$ in this neighborhood such that $e^{\ii\alpha}\in\mathbb Q(\ii)$ is not a root of unity. Such points are dense by the rational parametrization in the proof of Theorem~\ref{thm:rationalblocks}, with $\Delta=1$. Then $\alpha/(2\pi)$ is irrational, all $\cos(j\alpha)$ are rational, and~\eqref{eq:fejerperturb} has strictly positive rational cosine coefficients in every degree $0,\ldots,k-1$.

The product satisfies $F_\alpha(0)=1$ and $F_\alpha((j-\ell)\alpha)=0$ for $0\leq j,\ell<k$ with $j\ne\ell$. The numbers $j\alpha/(2\pi)$ are distinct modulo one because $\alpha/(2\pi)$ is irrational. Thus the polynomial satisfies the hypotheses of Lemma~\ref{lem:frame}, with $\Delta=1$ and angle parameters $\alpha_j=j\alpha$. Apply that lemma and Lemma~\ref{lem:completion}. The resulting rational skew-symmetric matrix has order at most $4k$. By~\eqref{eq:framesupport}, the eigenvalue support of the source for its Hermitian Hamiltonian is
\[
 \{-(k-1),\ldots,-1,0,1,\ldots,k-1\}.
\]
After clearing denominators, retain the connected component of the source and apply Theorem~\ref{thm:selection} with $\ell=2$, selecting exactly the $k$ prescribed coordinates. The eigenvalue-support formula~\eqref{eq:liftsupport} gives a positive rescaling of $\{-(2k-2),\ldots,2k-2\}$, which has $4k-3$ elements. Restriction to the source component does not change its eigenvalue support. Lemma~\ref{lem:labels} gives the same eigenvalue support at every vertex in the class. The differences between normalized transfer times remain irrational, since all transfer times and the minimum vertex period are divided by the same factor. Lemma~\ref{lem:stabilizer} excludes the stated switching automorphisms.
\end{proof}
For $k\geq3$, Corollary~\ref{thm:main} gives multiple state transfer without switching automorphisms relating distinct vertices of the class, answering the existence question in~\cite[Section~8]{GL}. These graphs have periodic vertices; global periodicity is not asserted. The number $4k-3$ is the size of the eigenvalue support in this construction. The final graph order can be much larger because of the denominators in the rational orthogonal change of coordinates.

\subsection{Simultaneous realization with prescribed periods}
Different complete classes can be joined without changing their prescribed transfers. Their periods, however, cannot be chosen independently. The same quadratic fields that restrict transfer times also determine the possible period ratios.

Fix an integer $m\geq1$. For $1\leq i\leq m$, let $\Lambda_i\subset\mathbb R/\mathbb Z$ be finite and contain zero, let $p_i>0$, and let $\varepsilon_i\in\{1,-1\}$. We seek classes with these normalized transfer times, minimum vertex periods proportional to $p_i$, and return phases $\varepsilon_i$. Put $\nu_i=1$ if $\varepsilon_i=1$, and $\nu_i=2$ if $\varepsilon_i=-1$, and use representatives in $[0,1)$.

\begin{theorem}\label{thm:simultaneous}
There is a connected oriented graph with pairwise disjoint complete PST classes $C_i=\T_X(a_i)$ at nonstationary periodic vertices, having normalized transfer-time sets $\Lambda_i$ relative to $a_i$, minimum vertex periods $T_i=\kappa p_i$ for one $\kappa>0$, and return phases $\varepsilon_i$, if and only if there are positive squarefree integers $\Delta_i$ such that
\begin{equation}\label{eq:jointperiodcondition}
 \begin{gathered}
 e^{2\pi\ii x/\nu_i}\in\mathbb Q(\sqrt{-\Delta_i})
 \qquad(x\in\Lambda_i),\\
 \frac{p_i\sqrt{\Delta_i}}{p_1\sqrt{\Delta_1}}\in\mathbb Q_{>0}
 \qquad(1\leq i\leq m).
 \end{gathered}
\end{equation}
\end{theorem}
\begin{proof}
Suppose such a realization exists. By Lemma~\ref{lem:quadratic}, write the gap generator at $a_i$ as $\delta_i=q_i\sqrt{\Delta_i}$ with $q_i\in\mathbb Q_{>0}$ and $\Delta_i$ positive and squarefree, allowing $\Delta_i=1$. Lemma~\ref{lem:field}, including~\eqref{eq:halftimefield} when $\varepsilon_i=-1$, gives the field condition. Also $T_i=2\pi/(q_i\sqrt{\Delta_i})=\kappa p_i$, so
\[
 \frac{p_i\sqrt{\Delta_i}}{p_1\sqrt{\Delta_1}}=\frac{q_1}{q_i}\in\mathbb Q_{>0}.
\]
This proves necessity.

Conversely, assume~\eqref{eq:jointperiodcondition} and put $b=p_1\sqrt{\Delta_1}$. Use the field $\mathbb Q(\sqrt{-\Delta_i})$ supplied by~\eqref{eq:jointperiodcondition} in the construction from the proof of Theorem~\ref{thm:schedules}. The rational model in~\eqref{eq:modelperiod}, before clearing denominators or selecting targets, has minimum vertex period $2\pi/(\nu_i\sqrt{\Delta_i})$ and return phase $\varepsilon_i$. Retain its source component, and multiply its matrix by the positive rational number
\[
 \frac{b}{\nu_i p_i\sqrt{\Delta_i}}.
\]
Call the resulting matrix $A_i$. Its distinguished vertices have minimum vertex period $2\pi p_i/b$, return phase $\varepsilon_i$, and the prescribed normalized transfer times.

Take the Cartesian cube of each $A_i$ and distinguish the constant tuples corresponding to its prescribed vertices. If $A_i$ has order $n_i$, then $n_i\geq2$, and at most $n_i$ of the $n_i^3$ coordinates are distinguished. Each block is connected and has nondistinguished coordinates. These coordinates allow the models to be joined while fixing every prescribed vertex. The distinguished transition entries in each block are cubes of the corresponding entries of $e^{tA_i}$, so both the times and their signs are preserved.

Join these blocks by rational orthogonal conjugations fixing all distinguished coordinates. For two disjoint connected blocks, choose one nondistinguished coordinate in each and rotate their plane by
\[
 \begin{pmatrix}3/5&-4/5\\4/5&3/5\end{pmatrix}.
\]
Each old incident edge survives with coefficient $3/5$ and gives a cross edge with coefficient $4/5$ or $-4/5$. There is no cancellation because the blocks were disjoint. Each chosen coordinate has a neighbor, so the new edges join the blocks while the old edges keep them connected. Iterating yields a connected rational skew-symmetric matrix. Conjugation fixes all transition entries between distinguished coordinates.

Clear denominators by an integer $D>0$ and apply Theorem~\ref{thm:selection} with $\ell=3$, selecting precisely those distinguished coordinates. The first cube supplied coordinates for joining; the cube in the selection theorem removes unwanted targets. For a final transfer within the $i$th prescribed class, the transition entry is therefore the ninth power of the corresponding entry of $e^{DrtA_i}$. Entries between different prescribed classes are zero. Hence all transfer and return times are divided by $Dr$, their signs are preserved, and~\eqref{eq:leakage} excludes every additional target. The minimum vertex periods are
\[
 T_i=\frac{2\pi}{bDr}\,p_i,
\]
as required.
\end{proof}

As in the proof of Theorem~\ref{thm:schedules}, diagonal sign conjugation permits independent prescriptions of the source-to-target signs in each class, at the chosen representatives of the normalized transfer times.

If return phases are not prescribed, the condition in~\eqref{eq:jointperiodcondition} with every $\nu_i=1$ is necessary for a mixed-graph realization and sufficient for an oriented realization. In particular, any finite family of admissible sets can be realized in one connected oriented graph when their periods are not specified: choose a field for each set and take $p_i=1/\sqrt{\Delta_i}$.

For example, consider two complete classes with normalized transfer times $\{0,1/4\}$ and $\{0,1/3\}$. Their fields are forced to be $\mathbb Q(\ii)$ and $\mathbb Q(\sqrt{-3})$, respectively. Their minimum vertex periods can therefore coexist precisely with ratios
\[
 T_2/T_1\in\frac1{\sqrt3}\mathbb Q_{>0},
\]
when a common scaling is allowed. In particular, equal or rationally related periods are impossible, even though both classes can occur in one connected graph.

\section{Explicit constructions}\label{sec:examples}
The following constructions illustrate the classification and target-selection results. The first uses the cell construction to realize specified complete classes. The second gives a globally periodic example with different first transfer times in the two directions.
\subsection{Complete classes of sizes five and six}\label{sec:six}
We start with an eight-vertex weighted model whose evolution cycles through six coordinate states. The cell construction then keeps either all six targets or only five of them.

Let $P_6\e_j=\e_{j+1\bmod6}$ for $0\le j<6$, let $\varepsilon=(1,-1,1,-1,1,-1)^{\mathsf T}$, and put
\begin{equation}\label{eq:sixseed}
 A_6=3(P_6-P_6^{\mathsf T})-(P_6^2-(P_6^{\mathsf T})^2),\qquad
 A_8=\begin{pmatrix}A_6&3\varepsilon\one_2^{\mathsf T}\\-3\one_2\varepsilon^{\mathsf T}&O_2\end{pmatrix}.
\end{equation}
The matrix $A_8$ is the skew adjacency matrix of an integer-weighted oriented graph of order eight. Its edge weights are one and three. We use it to construct the unweighted graphs below.

\begin{proposition}\label{prop:sixclock}
The first six vertices form a complete PST class for $H=\ii A_8$. Each has minimum vertex period $T_0=\pi/\sqrt3$, and, with $\tau_0=T_0/6$,
\begin{equation}\label{eq:sixexp}
 e^{\tau_0 A_8}=\diag\left(P_6,\begin{pmatrix}0&-1\\-1&0\end{pmatrix}\right).
\end{equation}
The minimum global period is also $T_0$.
\end{proposition}
\begin{proof}
Use the Fourier basis of the cyclic group of order six; see~\cite[Chapter~7, \S1]{SS}. On a Fourier vector for $P_6$ with eigenvalue $e^{\ii\theta}$, the corresponding eigenvalue of $A_6$ is
\[
 \ii(6\sin\theta-2\sin2\theta).
\]
For $\theta=0,\pi/3,2\pi/3,\pi,4\pi/3,5\pi/3$, these frequencies are respectively
\[
 0,\ 2\sqrt3,\ 4\sqrt3,\ 0,\ -4\sqrt3,\ -2\sqrt3.
\]
The added columns couple only the alternating Fourier vector $\varepsilon/\sqrt6$ to the uniform auxiliary vector $\one_2/\sqrt2$. The skew block on their span has frequency $6\sqrt3$. The auxiliary difference vector is a zero eigenvector. At time $\tau_0$, the nonalternating Fourier vectors acquire exactly their eigenvalues under $P_6$, while the coupled two-dimensional block is $-I$. The auxiliary difference vector stays fixed. This proves~\eqref{eq:sixexp}. The spectrum of $A_8$ is
\[
 0\ \text{(twice)},\quad \pm2\ii\sqrt3,\quad
 \pm4\ii\sqrt3,\quad\pm6\ii\sqrt3.
\]
The zero eigenvalue and frequency $2\sqrt3$ show that the minimum global period is $T_0$.

At any of the first six vertices the return amplitude, with $\vartheta=2\sqrt3t$, is
\begin{equation}\label{eq:sixreturn}
 f_6(t)=\frac{1+2\cos\vartheta+2\cos2\vartheta+\cos3\vartheta}{6}
     =\frac{(1+\cos\vartheta)(4\cos^2\vartheta-1)}6.
\end{equation}
The positive spectral weights at zero and frequency $2\sqrt3$ give minimum vertex period $T_0$. In $[0,T_0)$ the only zeros of~\eqref{eq:sixreturn} are $jT_0/6$, $1\leq j\leq5$. At these times, the corresponding integer powers of~\eqref{eq:sixexp} give the other five vertices in the class as the respective targets. Any transfer to a distinct vertex must have zero return amplitude, so no other target is possible. 
\end{proof}

\begin{corollary}\label{cor:unweightedsix}
There are connected oriented graphs with complete PST classes having the following parameters. The periods listed are minimum vertex periods.
\begin{center}
\begin{tabular}{@{}cccc@{}}
\toprule
\textup{Graph order}&\textup{Class size}&\textup{Period}&\textup{Normalized transfer times}\\
\midrule
$400$&$6$&$\pi/\sqrt3$&$0,1/6,2/6,3/6,4/6,5/6$\\
$1315$&$5$&$\pi/(2\sqrt3)$&$0,1/6,2/6,3/6,4/6$\\
\bottomrule
\end{tabular}
\end{center}
\end{corollary}
\begin{proof}
Let $A_\square=A_8\otimes I_8+I_8\otimes A_8$, and index its cells by $(a,b)$ with $0\leq a,b<8$. The complete PST class of $(0,0)$ in $\ii A_\square$ is exactly $\{(j,j):0\leq j<6\}$, since a tensor-product column is a coordinate state exactly when both factors are coordinate states. The same tensor factorization shows that the minimum vertex period at $(0,0)$ is $T_0$.

Apply Lemma~\ref{lem:cells} with $r=1$ and the square-root cell sizes $s_{ab}$ in~\eqref{eq:sizes6}, or with $r=2$ and those in~\eqref{eq:sizes5}. Appendix~\ref{app:data} checks every required cell condition and the respective orders $400$ and $1315$. Among the six quotient targets, precisely six and five cells are singletons. The lemma therefore gives the asserted complete PST classes, periods, and normalized transfer times. The additional singleton cells in the first construction are not quotient targets.

\end{proof}

The five-point set of normalized transfer times is a six-point grid with one point removed, so its translation stabilizer is trivial. Lemma~\ref{lem:stabilizer} therefore forces every switching automorphism preserving the five-vertex class to fix its vertices individually.

The orders in Corollary~\ref{cor:unweightedsix} are upper bounds for the minimum orders of such graphs. The construction does not establish global periodicity of either lifted graph.

\subsection{A globally periodic graph with a class of size two}\label{sec:nine}
We now construct a globally periodic graph with a complete PST class of size two and first normalized transfer time $1/3$. Start with a graph whose evolution cycles through three coordinate states. An orthogonal conjugation will fix two of those states and spread the third over four coordinates. We then use the return amplitude to exclude every other possible target.

Let
\[
 A_3=\begin{pmatrix}0&-1&1\\1&0&-1\\-1&1&0\end{pmatrix},\qquad
 A_{\mathrm{prod}}=A_3\otimes I_3+I_3\otimes A_3.
\]
Index the vertices of $A_{\mathrm{prod}}$ by $0,\ldots,8$, with $3a+b$ corresponding to $(a,b)$. Put $\mathcal I=\{1,2,7,8\}$, in this order, and $\xi=(1,-1,-1,-1)^\mathsf T$. Define the real symmetric matrix $R_Z$ by letting it be the identity outside $\mathcal I$, with zero cross blocks, and setting
\[
 R_Z[\mathcal I,\mathcal I]=\frac12\xi\xi^{\mathsf T}-I_4.
\]
Since $\xi^{\mathsf T}\xi=4$, we have $R_Z^2=I$. Set $A_Z=R_ZA_{\mathrm{prod}}R_Z$. Direct multiplication gives
\[
\renewcommand{\arraystretch}{1.04}
 A_Z=\begin{pmatrix}
0&0&0&-1&0&0&1&1&1\\
0&0&-1&0&0&0&0&1&0\\
0&1&0&0&1&1&0&0&1\\
1&0&0&0&-1&1&-1&0&0\\
0&0&-1&1&0&-1&0&0&-1\\
0&0&-1&-1&1&0&0&0&1\\
-1&0&0&1&0&0&0&1&-1\\
-1&-1&0&0&0&0&-1&0&-1\\
-1&0&-1&0&1&-1&1&1&0
\end{pmatrix}.
\]
All entries of $A_Z$ lie in $\{0,1,-1\}$, so it is the skew adjacency matrix of an oriented graph $Z$. Its degrees, in order, are $(4,2,4,4,4,4,4,4,6)$. The underlying graph is connected: for example the edges $03,06,07,08,71,12,24,25$ form a spanning tree. Conjugations by reflections are standard in graph switching, including the mixed and gain settings~\cite{BBCD}; the particular matrix above is verified explicitly and does not require any general switching criterion.

For this graph, write $H_Z=\ii A_Z$ and put
\[
 T=\frac{2\pi}{\sqrt3},\qquad \tau=\frac T3.
\]

\begin{theorem}\label{thm:nine}
The graph $Z$ has minimum global period $T$. Its complete PST class at $0$ is $\{0,4\}$, with minimum vertex period $T$. The first transfer times from $0$ to $4$ and from $4$ to $0$ are $\tau=T/3$ and $2\tau=2T/3$, respectively.
\end{theorem}
\begin{proof}
The eigenvalues of $A_3$ are $0,\pm\ii\sqrt3$, and
\[
 \bigl(e^{tA_3}\bigr)_{j,j}
    =f_3(t):=\frac{1+2\cos(\sqrt3t)}3.
\]
At $\tau=2\pi/(3\sqrt3)$, $e^{\tau A_3}$ is the cyclic permutation taking $0$ to $1$, $1$ to $2$, and $2$ to $0$. Therefore $e^{\tau A_{\mathrm{prod}}}$ sends $0$ to $4$, $4$ to $8$, and $8$ to $0$. Since $R_Z\e_0=\e_0$ and $R_Z\e_4=\e_4$, the identity $e^{tA_Z}=R_Ze^{tA_{\mathrm{prod}}}R_Z$ gives
\begin{align}
 U_Z(\tau)\e_0&=\e_4,\label{eq:forward9}\\
 U_Z(\tau)\e_4&=R_Z\e_8
   =\frac12(-\e_1+\e_2+\e_7-\e_8),\label{eq:diffuse9}\\
 U_Z(2\tau)\e_4&=\e_0.\label{eq:reverse9}
\end{align}
It also gives $U_Z(2\tau)\e_0=R_Z\e_8$.

Orthogonal conjugation preserves the spectrum. The distinct eigenvalues of $H_Z$ are
\[
 0,\quad\pm\sqrt3,\quad\pm2\sqrt3,
\]
with multiplicities $3,2,2,1,1$, respectively, so $U_Z(T)=I$. The simultaneous presence of $0$ and $\sqrt3$ implies that every global period is an integer multiple of $T$.

The return amplitude at $0$ is exactly
\begin{equation}\label{eq:return9}
 U_Z(t)_{0,0}=f_3(t)^2.
\end{equation}
Since $-1/3\leq f_3(t)\leq1$, we have $|U_Z(t)_{0,0}|=1$ precisely when $\cos(\sqrt3t)=1$. Thus the minimum vertex period is also $T$. For a transfer from $0$ to a distinct vertex the return amplitude must vanish. In $[0,T)$ the only zeros of~\eqref{eq:return9} are $\tau$ and $2\tau$. At the first time the state is $\e_4$; at the second it is the four-term state $R_Z\e_8$. There are no other PST targets, and the first transfer from $0$ to $4$ occurs at $T/3$. Lemma~\ref{lem:labels} then gives all transfer times in the reverse direction as $t/T\equiv2/3\pmod1$, so the first reverse transfer occurs at $2T/3$.

\end{proof}

Equation~\eqref{eq:diffuse9} shows that the reverse amplitude at $T/3$ is zero. The set of normalized transfer times $\{0,1/3\}$ has trivial translation stabilizer, so Lemma~\ref{lem:stabilizer} excludes a switching automorphism mapping $0$ to $4$. Theorem~\ref{thm:nine} therefore gives a negative answer to~\cite[Problem 35]{SL}: the first normalized transfer time is rational, but reverse transfer does not occur at that time, even under global periodicity.

\section{Concluding remarks and open questions}\label{sec:discussion}
Theorem~\ref{thm:schedules} classifies normalized transfer times of complete PST classes at nonstationary periodic vertices in mixed graphs: the condition $\cos(2\pi(x-y))\in\mathbb Q$ for all $x,y\in\Lambda$ is necessary and sufficient, and every admissible set has a connected oriented realization. The same construction determines the possible return phases in oriented realizations and selects exactly the required targets. Its consequences include arbitrary finite class sizes, sharp rational-time partition bounds, and simultaneous realization with prescribed relative minimum vertex periods and return phases. The realizations of admissible sets guarantee periodicity at the prescribed vertices. Further requirements on the realizing graph lead to the following questions.

\begin{question}\label{prob:global}
Can complete PST classes have arbitrarily large size in connected globally periodic mixed graphs?
\end{question}

The construction in the proof of Theorem~\ref{thm:schedules} establishes periodicity at the vertices of the prescribed class, but it does not establish global periodicity. The nine-vertex graph in Theorem~\ref{thm:nine} shows that a rational first normalized transfer time need not give transfer in both directions at that time, even under global periodicity. It does not resolve the size question in Question~\ref{prob:global}.

\begin{question}\label{prob:order}
Given a finite set $\Lambda\subset\mathbb R/\mathbb Z$ containing zero and satisfying condition~\textup{(iv)} of Theorem~\ref{thm:schedules}, what is the minimum order of a connected mixed graph realizing $\Lambda$ as the set of normalized transfer times of a complete PST class? What are the minimum orders of connected mixed graphs with complete PST classes of sizes five and six?
\end{question}

Corollary~\ref{cor:unweightedsix} gives the upper bounds $1315$ and $400$, respectively, using oriented graphs. Corollary~\ref{thm:main} gives an eigenvalue support of size $4k-3$ for a class of size $k$, but the denominators in the construction can make the graph order much larger.

\begin{question}\label{prob:structure}
Which finite sets $\Lambda\subset\mathbb R/\mathbb Z$ containing zero and satisfying condition~\textup{(iv)} of Theorem~\ref{thm:schedules} can be realized by connected regular mixed graphs? Which can be realized by connected oriented Cayley graphs?
\end{question}

The general construction does not preserve regularity or Cayley graph structure. In an oriented Cayley graph, vertex transitivity implies that the automorphisms preserving a complete class act transitively on that class. Lemma~\ref{lem:stabilizer} and Corollary~\ref{cor:symmetrybound} therefore force
\[
 \Lambda=\tfrac1k\mathbb Z/\mathbb Z,\qquad k\in\{1,2,3,4,6\}.
\]
Moreover, its skew adjacency matrix satisfies $A\one=0$, so zero lies in every vertex eigenvalue support and the return phase is $+1$. Thus the Cayley part of Question~\ref{prob:structure} concerns only these cyclic grids.

\appendix
\section{Explicit cell data and verification}\label{app:data}
The arrays below specify the square roots of the cell sizes used in Corollary~\ref{cor:unweightedsix}, indexed by $(a,b)$ for $0\leq a,b<8$. Write $A_\square=A_8\otimes I_8+I_8\otimes A_8$, with $A_8$ given in~\eqref{eq:sixseed}. The Hermitian quotient before time rescaling is $\ii A_\square$.

For the six-vertex class use $r=1$ and square-root cell sizes given by the following array (row $a$, column $b$):
\begin{equation}\label{eq:sizes6}
 (s_{ab})=\begin{pmatrix}
1&3&1&3&1&3&3&3\\
3&1&3&1&3&1&3&3\\
1&3&1&3&1&3&3&3\\
3&1&3&1&3&1&3&3\\
1&3&1&3&1&3&3&3\\
3&1&3&1&3&1&3&3\\
3&3&3&3&3&3&1&1\\
3&3&3&3&3&3&1&1
\end{pmatrix}.
\end{equation}
For the five-vertex class use $r=2$ and
\begin{equation}\label{eq:sizes5}
 (s_{ab})=\begin{pmatrix}
1&6&2&3&2&6&6&6\\
6&1&6&2&3&2&6&6\\
2&6&1&6&2&3&6&6\\
3&2&6&1&6&2&6&6\\
2&3&2&6&1&6&6&6\\
6&2&3&2&6&2&3&3\\
6&6&6&6&6&3&2&2\\
6&6&6&6&6&3&2&2
\end{pmatrix}.
\end{equation}
The cell conditions can be checked directly from these arrays, without constructing either large graph. In~\eqref{eq:sizes6}, an edge of weight one has $s_u=s_v$, and an edge of weight three has at least one endpoint with square-root size three; its possible unordered pairs $(s_u,s_v)$ are $(1,3)$ and $(3,3)$. Thus both degrees in~\eqref{eq:cellconditions} are integers and satisfy the capacity inequality. In~\eqref{eq:sizes5}, the possible unordered pairs $(s_u,s_v)$ are
\[
\begin{array}{c|l}
 w=1 &(1,2),(2,2),(3,6),(6,6)\\
 w=3 &(1,6),(2,3),(2,6),(3,3),(3,6),(6,6).
\end{array}
\]
Every pair satisfies~\eqref{eq:cellconditions} with $r=2$. These lists are obtained by changing one coordinate along a nonzero entry of~\eqref{eq:sixseed}; the supplied integer certificate checks the same list edge by edge.

The first array has $22$ entries one and $42$ entries three, so its graph has $22+42\cdot9=400$ vertices. The second has $5$ entries one, $17$ entries two, $10$ entries three, and $32$ entries six. Its graph therefore has order $5+17\cdot4+10\cdot9+32\cdot36=1315$.

The edge counts are $4896$ and $30376$, respectively. Let $E_\square$ be the edge set of the underlying graph of $A_\square$. Counting each unordered edge once gives
\[
 |E(Y)|=\sum_{\{u,v\}\in E_\square} r|(A_\square)_{u,v}|s_us_v.
\]
The supplement supplies the corresponding sparse skew adjacency matrices and a reproducible exact verification script. Multiplication by $\ii$ gives the Hermitian adjacency matrices.

Write $A_Y$ for the real skew adjacency matrix of either constructed graph $Y$, and let $Q$ be the normalized characteristic matrix of its cells, as in Lemma~\ref{lem:cells}. The accompanying scripts verify entries in $\{0,1,-1\}$, zero diagonal, skew-symmetry, connectivity, and the integer identities
\[
 W^{\mathsf T}W=L_s^2I,\qquad A_YW=rWA_\square,
 \qquad W=L_s Q,\quad L_s=\operatorname{lcm}_v s_v.
\]
They also verify the two explicit matrix constructions, rational orbit frames for both return phases, odd-power lifts, cosine polynomials of small degree, partitions under $\sim_{\mathbb Q}$ for both return phases, and joined models with prescribed relative minimum vertex periods.


\begin{thebibliography}{99}
\bibitem{AC}
A. Acuaviva, A. Chan, S. Eldridge, C. Godsil, M. How-Chun-Lun, C. Tamon, E. Wright, and X. Zhang,
State transfer in chiral quantum walks,
\emph{Philosophical Transactions of the Royal Society A} \textbf{383} (2025), 20240420.
\href{https://doi.org/10.1098/rsta.2024.0420}{doi:10.1098/rsta.2024.0420}.

\bibitem{BFF}
R. Bachman, E. Fredette, J. Fuller, M. Landry, M. Opperman, C. Tamon, and A. Tollefson,
Perfect state transfer on quotient graphs,
\emph{Quantum Information and Computation} \textbf{12} (2012), 293--313.
\href{https://doi.org/10.26421/QIC12.3-4-9}{doi:10.26421/QIC12.3-4-9}.

\bibitem{BBCD}
F. Belardo, M. Brunetti, M. Cavaleri, and A. Donno,
Godsil--McKay switching for mixed and gain graphs over the circle group,
\emph{Linear Algebra and its Applications} \textbf{614} (2021), 256--269.
\href{https://doi.org/10.1016/j.laa.2020.04.025}{doi:10.1016/j.laa.2020.04.025}.

\bibitem{CPRS}
A. Chan, V. R. T. Pantangi, A. S. Razafimahatratra, and P. Sin,
Perfect and multiple state transfer in oriented Cayley graphs,
\href{https://arxiv.org/abs/2510.09873v2}{arXiv:2510.09873v2}, 2026.

\bibitem{CDEL}
M. Christandl, N. Datta, A. Ekert, and A. J. Landahl,
Perfect state transfer in quantum spin networks,
\emph{Physical Review Letters} \textbf{92} (2004), 187902.
\href{https://doi.org/10.1103/PhysRevLett.92.187902}{doi:10.1103/PhysRevLett.92.187902}.

\bibitem{GPeriodic}
C. Godsil,
Periodic graphs,
\emph{The Electronic Journal of Combinatorics} \textbf{18}(1) (2011), \#P23.
\href{https://doi.org/10.37236/510}{doi:10.37236/510}.

\bibitem{GReal}
C. Godsil,
Real state transfer,
\href{https://arxiv.org/abs/1710.04042}{arXiv:1710.04042}, 2017.

\bibitem{GL}
C. Godsil and S. Lato,
Perfect state transfer on oriented graphs,
\emph{Linear Algebra and its Applications} \textbf{604} (2020), 278--292.
\href{https://doi.org/10.1016/j.laa.2020.06.025}{doi:10.1016/j.laa.2020.06.025}.

\bibitem{GM}
K. Guo and B. Mohar,
Hermitian adjacency matrix of digraphs and mixed graphs,
\emph{Journal of Graph Theory} \textbf{85}(1) (2017), 217--248.
\href{https://doi.org/10.1002/jgt.22057}{doi:10.1002/jgt.22057}.

\bibitem{HW}
G. H. Hardy and E. M. Wright,
\emph{An Introduction to the Theory of Numbers}, 6th ed.,
revised by D. R. Heath-Brown and J. H. Silverman,
Oxford University Press, Oxford, 2008.
\href{https://doi.org/10.1093/oso/9780199219858.001.0001}{doi:10.1093/oso/9780199219858.001.0001}.

\bibitem{Householder}
A. S. Householder,
Unitary triangularization of a nonsymmetric matrix,
\emph{Journal of the ACM} \textbf{5}(4) (1958), 339--342.
\href{https://doi.org/10.1145/320941.320947}{doi:10.1145/320941.320947}.

\bibitem{MTZ}
H. Miki, S. Tsujimoto, and D. Zhao,
Quantum walks on simplexes and multiple perfect state transfer,
\emph{Journal of Physics A: Mathematical and Theoretical} \textbf{56} (2023), 465201.
\href{https://doi.org/10.1088/1751-8121/ad034a}{doi:10.1088/1751-8121/ad034a}.

\bibitem{Milne}
J. S. Milne,
\emph{Fields and Galois Theory}, version 5.10, 2022.
\url{https://www.jmilne.org/math/CourseNotes/FT.pdf}.

\bibitem{SFO}
P. Serra, A. Ferr\'on, and O. Osenda,
Engineering perfect state transfer graphs via Givens transformations,
\href{https://arxiv.org/abs/2604.02536}{arXiv:2604.02536}, 2026.

\bibitem{SL}
X. Song and H. Lin,
State transfer on mixed graphs,
\emph{The Electronic Journal of Combinatorics} \textbf{33}(3) (2026), \#P3.60.
\href{https://doi.org/10.37236/14214}{doi:10.37236/14214}.

\bibitem{SS}
E. M. Stein and R. Shakarchi,
\emph{Fourier Analysis: An Introduction},
Princeton Lectures in Analysis, vol.~1,
Princeton University Press, Princeton, NJ, 2003.
\href{https://press.princeton.edu/books/hardcover/9780691113845/fourier-analysis}{Publisher's page}.
\end{thebibliography}
\end{document}